\documentclass[12pt]{amsart}
\usepackage{amsmath,amsthm,amsfonts,amssymb,latexsym,verbatim,bbm}
\input xy
\xyoption{all}
\usepackage{hyperref,multirow}
\usepackage[shortlabels]{enumitem}

\allowdisplaybreaks[2] \textwidth15.1cm \textheight22cm \headheight12pt \oddsidemargin.4cm
\def\mhline{\noalign{\ifnum0=`}\fi\hrule height 4\arrayrulewidth \futurelet
   \@tempa\oxhline}
\def\oxhline{\ifx\@tempa\hline\vskip \doublerulesep\fi
      \ifnum0=`{\fi}}

\makeatother

\theoremstyle{plain}
\newtheorem{theorem}{Theorem}[section]
\newtheorem{thm}[theorem]{Theorem}
\newtheorem{lemma}[theorem]{Lemma}

\newtheorem{prop}[theorem]{Proposition}
\newtheorem{example}[theorem]{Example}

\newtheorem{defn}[theorem]{Definition}

\newtheorem{cor}[theorem]{Corollary}

\newtheorem{question}[theorem]{Question}
\newtheorem{problem}[theorem]{Problem}
\newtheorem{rmk}[theorem]{Remark}

\newcommand{\iso}{\cong}
\newcommand{\arr}{\rightarrow}

\newcommand{\R}{\mathbb{R}}
\newcommand{\C}{\mathbb{C}}
\newcommand{\Z}{\mathbb{Z}}

\newcommand{\mcA}{\mathcal{A}}

\newcommand{\mcI}{\mathcal{I}}

\newcommand{\mcM}{\mathcal{M}}

\DeclareMathOperator{\vspan}{span}

\DeclareMathOperator{\Der}{Der}
\DeclareMathOperator{\nbc}{nbc}
\DeclareMathOperator{\codim}{codim}
\DeclareMathOperator{\al}{al}
\DeclareMathOperator{\fl}{fl}

\begin{document}

\title[Rationally smooth Schubert varieties and inversion arrangements]{Rationally smooth Schubert varieties and inversion hyperplane arrangements}

\author{William Slofstra}
\email{weslofst@uwaterloo.ca}

\begin{abstract}
    We show that an element $w$ of a finite Weyl group $W$ is rationally smooth
    if and only if the hyperplane arrangement $\mcI$ associated to the
    inversion set of $w$ is inductively free, and the product $(d_1+1) \cdots
    (d_l+1)$ of the coexponents $d_1,\ldots,d_l$ is equal to the size of the
    Bruhat interval $[e,w]$, where $e$ is the identity in $W$. As part of the
    proof, we describe exactly when a rationally smooth element in a finite
    Weyl group has a chain Billey-Postnikov decomposition. For finite Coxeter
    groups, we show that chain Billey-Postnikov decompositions are connected
    with certain modular coatoms of $\mcI$. 
\end{abstract}
\maketitle

\section{Introduction}
A central hyperplane arrangement in a Euclidean space is free if its module of
derivations is free, in which case the the degrees of a homogeneous basis are
called the coexponents of the arrangement. Let $Q_{\mcA}(t)$ be the Poincare
polynomial of an arrangement $\mcA$. If $\mcA$ is free with coexponents
$d_1,\ldots,d_l$, then $Q_{\mcA}(t) = \prod_i (1 + d_i t)$ by a result of Terao
\cite{Te81}. The evaluation of $Q_{\mcA}(t)$ at $t=1$ is the number of chambers
of $\mcA$. 

Let $W$ be a finite Coxeter group with length function $\ell$. It is well-known
that the Poincare polynomial $P(q) = \sum_{x \in W} q^{\ell(x)}$ factors as
$P(q) = \prod_i [m_i + 1]_q$, where $[m]_q$ is the $q$-integer $1 + q + q^2 +
\ldots + q^{m-1}$ and $m_1,\ldots,m_{l}$ are the exponents of $W$. Let $R$ be a
root system associated to $W$, and let $\mcA$ be the Coxeter hyperplane
arrangement cut out by the elements of $R$. Then $\mcA$ is free, and
$Q_{\mcA}(t) = \prod_i (1 + m_i t)$, where the coexponents $m_1, \ldots, m_l$
are equal to the exponents of $W$ \cite{Ar79, Sa75, OT92}.  More recently,
Barakat and Cuntz have shown that $\mcA$ is inductively free \cite{BC12}.  The
Poincare polynomial $Q_{\mcA}(t)$ can be written strictly in terms of the Weyl
group as $\sum_{w \in W} t^{\ell'(w)}$, where $\ell'(w)$ is the absolute length
of $w$ \cite{Ca72} \cite{ST54}. The identity $Q_{\mcA}(1) = \prod_i (1+m_i) =
P(1)$ recovers the fact that the number of chambers of $\mcA$ is equal to the
size of $W$. 

In this paper we show that if $W$ is a Weyl group, then this classical picture
of the relationship between $W$ and $\mcA$ extends to---and in fact
characterizes---the rationally smooth elements of $W$. The Poincare polynomial
of an element $w \in W$ is $P_w(q) = \sum_{x \in [e,w]} q^{\ell(x)}$, where
$[e,w]$ is the interval between the identity $e$ and $w$ in Bruhat order. If
$X(w)$ is a Schubert variety indexed by $w$, then $P_w(q^2) = \sum_i q^i \dim
H^i(X(w))$. A theorem of Carrell-Peterson states that $X(w)$ is rationally
smooth if and only if $P_w(q)$ is palindromic, meaning that the coefficients of
$P_w(q)$ are the same whether read from top-degree to bottom-degree, or
vice-versa. We say that $w \in W$ is rationally smooth if this condition is
satisfied. In finite type, it follows from results primarily of Gasharov
\cite{Ga98}, Billey \cite{Bi98}, Billey-Postnikov \cite{BP05}, and
Akyildiz-Carrell \cite{AC12} that $P_w(q)$ factors as a product $\prod_i
[m_i+1]_q$ of $q$-integers for every rationally smooth element $w$. The
integers $m_1,\ldots,m_l$ arising in this factorization can be uniquely
determined from $P_w(q)$. We call these integers the exponents of the
rationally smooth element $w$ (see Theorem \ref{T:exponents} and Definition
\ref{D:exponents}).

The inversion hyperplane arrangement $\mcI(w)$ of an element $w \in W$ is the
arrangement cut out by the inversion set $I(w)$ of $w$. If $w_0$ is the maximal
element of $W$, then $\mcI(w)$ is the Coxeter arrangement cut out by $R$, and
$P_{w_0}(q)$ is the Poincare polynomial $P(q)$. For a general element $w$,
Hultman, Linusson, Shareshian, and Sj\"ostrand (type A, \cite{HLSS09}) and
Hultman (all Coxeter groups, \cite{Hu11}) have shown that if $Q_{\mcI(w)}(1)$
is equal to the size of the Bruhat interval $[e,w]$, then
\begin{equation}\label{E:HLSS}
    Q_{\mcI(w)}(t) = \sum_{u \in [e,w]} t^{\al(u,w)} = \sum_{u \in [e,w]}
        t^{\ell'(u w^{-1})},
\end{equation}
where $\al(u,w)$ is the distance from $u$ to $w$ in the Bruhat graph of
$[e,w]$. In addition, \cite{HLSS09} and \cite{Hu11} show that Equation
(\ref{E:HLSS}) holds if $w$ is rationally smooth. We say that $w$ satisfies the
HLSS condition if $Q_{\mcI(w)}(1)$ is equal to the size of $[e,w]$. In type
$A$, the HLSS condition holds for $w$ if and only if the Schubert variety
$X(w)$ is defined by inclusions \cite{HLSS09}, a condition introduced by
Gasharov and Reiner which now has a number of equivalent formulations
\cite{GR02} \cite{Sj07} \cite{LM14}. 

In type $A$, Oh, Postnikov, and Yoo show that if $w$ is rationally smooth, then
$Q_{\mcI(w)}(t) = \prod_i (1+m_i t)$, where $m_1,\ldots,m_l$ are the exponents
of $w$ \cite{OPY08}. Their proof implicitly implies that $\mcI(w)$ is free. Yoo
has conjectured that $Q_{\mcI(w)}(t)$ factors in this way for rationally smooth
elements in all finite Weyl groups \cite[Conjecture 1.7.3]{Yoo11}. 
Oh, Postnikov, and Yoo also show that, in type $A$, if $w$ is rationally smooth
then the Poincare polynomial $P_w(q)$ is equal to the rank-generating polynomial
for the poset of regions of $\mcI(w)$.\footnote{The rank-generating polynomial is
always palindromic, so in fact $w$ is rationally smooth if and only if $P_w(q)$
is equal to the rank-generating polynomial of $\mcI(w)$.} This result has been
extended to all finite-type Weyl groups by Oh and Yoo \cite{OY10}, using what
we will call chain Billey-Postnikov (BP) decompositions (this is a modest
variation on the terminology in \cite{OY10}).

Inspired by \cite{OY10}, we list the rationally smooth elements in finite type
which do not have a chain BP decomposition. We then show that an element $w$ of
an arbitrary finite Coxeter group has a chain BP decomposition if and only if
$\mcI(w)$ has a modular coatom of a certain form. Using these two results, we
prove Yoo's conjecture by showing that $\mcI(w)$ is inductively free when $w$
is rationally smooth, with coexponents equal to the exponents of $w$. Using the
root system pattern avoidance criterion for rational smoothness due to Billey
and Postnikov \cite{BP05}, we prove a converse: $w$ is rationally smooth if and
only if $\mcI(w)$ is free and $w$ satisfies the HLSS condition. Note that
when $\mcI(w)$ is free, $w$ satisfies the HLSS condition if and only if
$\prod_i (1+d_i) = |[e,w]|$. Finally, we show that $w$ has a complete chain BP
decomposition if and only if $w$ is rationally smooth and $\mcI(w)$ is
supersolvable. 

\subsection{Organization}
We start in Section \ref{S:background} by giving some additional background
and terminology. The main results are stated in Section \ref{S:main}. 
Section \ref{S:bp} contains results on chain BP decompositions. In Section
\ref{S:hultman} we give some more background on the HLSS condition. Section
\ref{S:modular} develops the connection between chain BP decompositions and
modular elements, and proves one direction of the main theorem. Section
\ref{S:flattening} discusses root-system pattern avoidance and the HLSS
condition. The proofs of the main theorems are finished in Subsection
\ref{SS:proof}. Examples and further directions are discussed in Section
\ref{S:further}.

\subsection{Acknowledgements}
I thank Jim Carrell for explaining his results, which inspired this project. I
thank Ed Richmond for conjecturing Proposition \ref{P:typeAfree}, as well as
many helpful discussions. I thank Suho Oh and Hwanchul Yoo for helpful
conversations. I thank the anonymous referee and Monica Vazirani for helpful
suggestions on the manuscript.

\section{Background and notation}\label{S:background}
In this section we go over some background and notation that will be used
throughout the paper.

\subsection{Root systems and Weyl groups}\label{SS:rsbackground} 

$W$ will be a finite Coxeter group with generating set $S$ and root system $R$.
We use $l = |S|$ to refer to the rank of $W$.

We let $\leq$ denote Bruhat order on $W$. The \emph{Bruhat graph} of $W$ is the
directed graph with vertex set $W$, and an edge from $w$ to $tw$ whenever $t$
is a reflection with $\ell(tw) > \ell(w)$. By definition, the distance
$\al(u,w)$ from $u$ to $w$ in the Bruhat graph is infinite unless $u \leq w$ in
Bruhat order. The \emph{absolute (or reflection) length} $\ell'(w)$ of $w \in
W$ is the smallest number $k$ such that $w$ is a product of $k$ reflections in
$W$.

Given a root $\alpha \in R$, we let $t_{\alpha}$ denote the corresponding
reflection. If $s \in S$ is a simple reflection, we let $\alpha_s$ denote the
corresponding simple positive root. We let $D_L(w)$ and $D_R(w)$ denote the left and
right descent sets of $w \in W$, and $S(w)$ denote the support of $w$ (i.e. the
set of all simple reflections which appear in some reduced expression for $w$).

Given a subset $J \subset S$, we let $W_J$ denote the parabolic subgroup
generated by $J$, $W^J$ denote the set of minimal length left coset
representatives, and ${}^J W$ denote the set of minimal length right coset
representatives. Every element $w \in W$ can be written uniquely as $w = uv$
where $v \in {}^J W$ and $u \in W_J$. This factorization is called the \emph{left
parabolic decomposition} of $w$. \emph{Right parabolic decompositions} are defined
similarly.

The \emph{Poincare polynomial} of $w$ relative to $J$ is ${}^J P_w(q) =
\sum_{x \in [e,w] \cap {}^J W} q^{\ell(x)}$. The Poincare polynomial $P^J_w(q)$
is defined similarly with ${}^J W$ replaced by $W^J$. As in the introduction,
we write $P_w(q)$ for $P^{\emptyset}_w(q)$. We say that a polynomial $P(q)$ of
degree $l$ is \emph{palindromic} if $q^l P(q^{-1}) = P(q)$, meaning again that
the coefficients of $P(q)$ are the same whether read from top-degree to
bottom-degree or vice-versa. If $P_w(q)$ is palindromic, we say that $w$ is
\emph{rationally smooth}.

We refer to \cite{BL00} for more background on the combinatorics of
rationally smooth Schubert varieties.

\subsection{Hyperplane arrangements}\label{SS:habackground}
A central hyperplane arrangement $\mcA$ is a union of linear hyperplanes in
some vector space $V$. For convenience, we will also use $\mcA$ to refer to the
set of hyperplanes which it contains.  Given a central hyperplane arrangement
$\mcA$ in $V$, we can choose a defining linear form $\alpha_H$ for each $H \in
\mcA$.  Then $\mcA$ is cut out by the polynomial $Q = \prod \alpha_H$, where
the product is over all $H \in \mcA$. The complexification $\mcA_{\C}$ is the
arrangement cut out by $Q$ in $V_{\C} = V \otimes \C$, and $\Der(\mcA)$ is the
$\C[V_{\C}]$-module of derivations of $\C[V_{\C}]$ which preserve the ideal
$\C[V_{\C}] Q$. The module $\Der(\mcA)$ is graded by polynomial degree. As in
the introduction, $\mcA$ is said to be \emph{free} if $\Der(\mcA)$ is a free
$\C[V_{\C}]$-module, in which case $\Der(\mcA)$ has a homogeneous basis. The
degrees of this basis are called the \emph{coexponents} of $\mcA$.

Let $L(\mcA)$ denote the intersection lattice of $\mcA$. By convention, the
maximal element of $L(\mcA)$ is the center $\bigcap_{H \in \mcA} H$ of $\mcA$.
An element of $L(\mcA)$ is called a \emph{flat}.  If $X \in L(\mcA)$, we let
$\mcA^X$ denote the restriction of $\mcA$ to $X$, and $\mcA_X$ the localization
of $\mcA$ at $X$ (the localization $\mcA_X$ consists of the hyperplanes $H \in
\mcA$ such that $X \subseteq H$). Given $H \in \mcA$, we let $\mcA \setminus H$
denote the deletion by $H$, which is the arrangement containing all hyperplanes
of $\mcA$ except $H$. Finally, if $Y \subset V$ is a subspace of the center of
$\mcA$, we let $\mcA / Y$ be the quotient of $\mcA$ by $Y$, an arrangement in
$V / Y$. The arrangement $\mcA$ is free if and only if $\mcA / Y$ is free. If
$\mcA$ is free, then $0$ is a coexponent of $\mcA$ of multiplicity at least
$\dim Y$, and the remaining coexponents of $\mcA$ are equal to the coexponents
of $\mcA / Y$. 

The addition theorem states that if $\mcA \setminus H$ is free with coexponents
$d_1,\ldots,d_{l-1}, d_{l}-1$, and $\mcA^H$ is free with coexponents
$d_1,\ldots,d_{l-1}$, then $\mcA$ is free with coexponents $d_1,\ldots,d_l$
\cite{OT92}. An arrangement is said to be \emph{inductively free} if either (a)
$\mcA$ contains no hyperplanes (in which case all exponents are zero), or (b)
there is some hyperplane $H \subset \mcA$ such that $\mcA \setminus H$ is
inductively free with coexponents $d_1,\ldots,d_l-1$, and $\mcA^{H}$ is
inductively free with coexponents $d_1,\ldots,d_{l-1}$. The addition theorem
implies that inductively free arrangements are free. 

Given an order on the set $\{\alpha_H : H \in \mcA \}$, a \emph{broken circuit} is
defined to be an ordered subset $\{\alpha_{H_1} < \ldots < \alpha_{H_k}\}$ such
that there is $\alpha_{H_{k+1}} > \alpha_{H_k}$ for which
$\{\alpha_{H_1},\ldots,\alpha_{H_{k+1}}\}$ is a minimal linearly dependent set
in $\{\alpha_H\}$.\footnote{The order is often reversed in this definition. We
are using the order convention from \cite{Hu11}.} An \emph{$\nbc$-set} is an
ordered subset $\{\alpha_{H_1} < \ldots < \alpha_{H_k}\}$ which does not
contain a broken circuit. The Orlik-Solomon algebra is the free exterior
algebra generated by the forms $\alpha_H$, modulo the relations
\begin{equation*}
    \partial \left(\alpha_{H_1} \wedge \ldots \wedge \alpha_{H_k}\right) = 0 \text{ if }
        \codim H_1 \cap \ldots \cap H_k < k.
\end{equation*}
Here $\partial$ is the unique derivation which sends $\alpha_{H} \mapsto 1$
for all $H \in \mcA$.  It is well-known that the elements
\begin{equation*}
    \alpha_{H_1} \wedge \ldots \wedge \alpha_{H_k} \text{ : } 
        \{\alpha_{H_1},\ldots,\alpha_{H_k}\} \text{ is an $\nbc$-set}
\end{equation*}
form a basis of this algebra, and that the Orlik-Solomon algebra is isomorphic
to the cohomology ring $H^*(\mcM(\mcA))$ of the complement $\mcM(\mcA) := V_{\C}
\setminus \mcA_{\C}$. In particular, the number of $\nbc$-sets does not depend
on the order on $\{\alpha_H\}$. 

The Poincare polynomial $Q_{\mcA}(t)$ of $\mcA$ can be defined as the Poincare
polynomial $\sum_i t^i H^i(\mcM(\mcA))$ of $\mcM(\mcA)$. The coefficient of
$t^i$ in $Q_{\mcA}(t)$ is then the number of $\nbc$-sets of size $i$. The
Poincare polynomial is related to the characteristic polynomial
$\chi_{\mcA}(t)$ by the identity $t^l Q_{\mcA}(-t^{-1}) = \chi_{\mcA}(t)$. 

\subsection{Inversion arrangements}
We now fix $V$ to be the ambient Euclidean space containing $R$. For
convenience we identify $V$ and $V^*$ using the Euclidean form. Let $R^+$ and
$R^-$ denote the positive and negative root sets of $R$, respectively. The
\emph{inversion set} of $w \in W$ is defined to be 
\begin{equation*}
    I(w) = \{\alpha \in R^+ : w^{-1} \alpha \in R^-\}.
\end{equation*}
The \emph{inversion hyperplane arrangement} of $w$ is 
\begin{equation*}
    \mcI(w) = \bigcup_{\alpha \in I(w)} \ker \alpha
\end{equation*}
in $V$. Note that since $V$ and $V^*$ are identified, $\ker \alpha = \{\beta
\in V : (\beta, \alpha) = 0\}$. 

\section{Main results}\label{S:main}

The starting point for this paper is the following theorem:
\begin{thm}[\cite{AC12}, \cite{Ga98}, \cite{Bi98}, \cite{BP05}]\label{T:exponents}
    Let $W$ be a finite Weyl group.  An element $w \in W$ is rationally smooth
    if and only if
    \begin{equation}\label{E:exponents}
        P_w(q) = \prod_{i=1}^{l} \left[ m_i+1\right]_q
    \end{equation}
    for some collection of non-negative integers $m_1,\ldots,m_l$.
\end{thm}
As in the introduction, here $[m]_q = (1+q+q^2+\ldots+q^{m-1})$.  The proof of
Theorem \ref{T:exponents} splits into cases, and is due to a variety of
authors. As far as the author is aware, the complete theorem has not been
previously stated before, so we give a full account of the proof in Section
\ref{S:bp}.  The integers $m_1,\ldots,m_l$ appearing in Theorem
\ref{T:exponents} are uniquely determined by Equation (\ref{E:exponents}),
leading to the following definition.
\begin{defn}\label{D:exponents}
    Let $W$ be a finite Weyl group.  If $w \in W$ is rationally smooth, then
    the \emph{exponents} of $w$ are the integers $m_1,\ldots,m_l$ appearing in
    Theorem \ref{T:exponents}.
\end{defn}
The main result of this paper is then:
\begin{thm}\label{T:main}
    Let $W$ be a finite Weyl group.  An element $w \in W$ is rationally smooth
    if and only if:
    \begin{enumerate}[(a)]
        \item the inversion hyperplane arrangement $\mcI(w)$ is free with
            coexponents $d_1,\ldots,d_l$, and
        \item the product $\prod_i (1+d_i)$ is equal to the size of the Bruhat
            interval $[e,w]$.  
    \end{enumerate}
    Furthermore, if $w$ is rationally smooth then the coexponents
    $d_1,\ldots,d_l$ are equal to the exponents of $w$. 
\end{thm}
As mentioned in the introduction, one direction of Theorem \ref{T:main} is
implicitly proved for type $A$ in \cite{OPY08}. In the proof of Theorem
\ref{T:main} we will show that if $w$ is rationally smooth, then $\mcI(w)$ is
in fact inductively free. As an immediate consequence of Theorem \ref{T:main},
we have that:
\begin{cor}\label{C:poincare}
    Let $W$ be a finite Weyl group.  If $w \in W$ is rationally smooth, then
    \begin{equation*}
        Q_{\mcI(w)}(t) = \prod_{i=1}^l (1+m_i t),
    \end{equation*}
    where $m_1,\ldots,m_l$ are the exponents of $w$.
\end{cor}
\begin{example}\label{Ex:easytypeA}
    Let $W = S_4$ be the Weyl group of type $A_3$, and let $S = \{s_i : 1 \leq i
    \leq 4\}$ be the standard generating set of simple transpositions, so $s_i = 
    (i\quad i+1)$ in disjoint cycle notation. If $w = s_2 s_3 s_2 s_1$ then 
    \begin{equation*}
        P_w(q) = 1 + 3q + 4q^2 + 3q^3 + q^4 = [2]_q^2 \cdot [3]_q,
    \end{equation*}
    and $w$ has exponents $1, 1, 2$. The inversion hyperplane arrangement
    $\mcI(w)$ is the arrangement in $\C^3$ cut out by the linear functions
    $\alpha_2$, $\alpha_3$, $\alpha_2 + \alpha_3$, and $\alpha_1 + \alpha_2 +
    \alpha_3$, where $\alpha_i$ is dual to the $i$th standard basis vector in $\C^3$. 
    Since $w$ is rationally smooth, Theorem \ref{T:main} and Corollary 
    \ref{C:poincare} imply that $\mcI(w)$ is free with coexponents $1,1,2$,
    and 
    \begin{equation*}
        Q_{\mcI(w)}(t) = (1+t)^2 (1+2t).
    \end{equation*} 
    On the other hand, if $w' = s_2 s_3 s_1 s_2$ then
    \begin{equation*}
        P_{w'}(q) = 1 + 3q + 5q^2 + 4q^3 + q^4,
    \end{equation*}
    so $w'$ is not rationally smooth. The Poincare polynomial
    \begin{equation*}
        Q_{\mcI(w')}(t) = (1+t)(1+3t+3t^2)
    \end{equation*}
    does not factor into linear factors, so $\mcI(w')$ is not free. 
\end{example}
    
We can also characterize when $w$ is rationally smooth and $\mcI(w)$ is not
just free, but supersolvable. 
\begin{defn}
    Let $w = uv$ be a left parabolic decomposition of $w$ with respect to $J
    \subset S$, so $u \in W_J$ and $v \in {}^J W$. Then $w = uv$ is a
    \emph{(left) Billey-Postnikov (BP) decomposition} if $P_w(q) = P_u(q)\cdot {}^J
    P_v(q)$. If in addition $[e,v] \cap {}^J W$ is a chain, we say that $w =
    uv$ is a \emph{(left) chain BP decomposition}. \emph{Right chain BP
    decompositions} are defined similarly.
\end{defn}
Note that a left BP decomposition of $w$ is a right BP decomposition of
$w^{-1}$. 
\begin{defn}
    We say that $w$ has a \emph{complete chain BP decomposition} if either
    $w$ is the identity, or $w$ has a left or right chain BP decomposition
    $w = uv$ or $w = vu$, where $u \in W_J$ has a complete chain
    BP decomposition.
\end{defn}
Let $W$ be an arbitrary finite Coxeter group.  The interval $[e,v] \cap {}^J W$
is a chain if and only if ${}^J P_v(q) = [\ell(v)+1]_q$, so if $w$ has a
complete chain BP decomposition, then $P_w(q)$ is a product of $q$-integers.
Hence if $w$ has a complete chain BP decomposition, then $w$ is rationally
smooth, and it is also possible to define exponents of $w$ as in Definition
\ref{D:exponents}. In types $A$, $B$, $C$, and $G_2$, every rationally smooth
element has a complete chain BP decomposition, but this is not true in types $D$, $E$,
or $F_4$ (see Theorem \ref{T:chainbpexist}). Recall that an arrangement $\mcA$
is supersolvable if and only if $L(\mcA)$ has a complete chain of modular
elements.
\begin{thm}\label{T:supersolv}
    Let $W$ be a finite Weyl group. An element $w \in W$ has a complete chain
    BP decomposition if and only if $w \in W$ is rationally smooth and
    $\mcI(w)$ is supersolvable.
\end{thm}
If $w = uv$ is a left parabolic decomposition, then the inversion set $I(w)$ is
the disjoint union of $I(u)$ and $u I(v)$, and
\begin{equation*}
    X = \bigcap \mcI(u) = \bigcap_{\alpha \in I(u)} \ker \alpha 
\end{equation*}
is a flat of $L(\mcI(w))$. The proofs of Theorems \ref{T:main} and
\ref{T:supersolv} are based on a connection between chain BP decompositions
and certain modular coatoms of the inversion arrangement:
\begin{thm}\label{T:bpmodular}
    Suppose that $w = uv$ is a left parabolic decomposition with respect to
    $J$, so $u \in W_J$ and $v \in {}^J W$. Let $X = \bigcap \mcI(u)$. Then
    $w = uv$ is a chain BP decomposition if and only if $X$ is a modular
    coatom of $L(\mcI(w))$.
\end{thm}
Theorem \ref{T:bpmodular} holds for arbitrary finite Coxeter groups.
As will be explained in Section \ref{S:modular}, the key point of Theorem
\ref{T:bpmodular} is that $[e,v] \cap {}^J W$ is a chain if and only if a
certain linear condition on $I(v)$ holds. Also note that $I(w^{-1}) = -w^{-1}
I(w)$, so $\mcI(w)$ and $\mcI(w^{-1})$ are linearly isomorphic, and Theorem
\ref{T:bpmodular} could equivalently be stated in terms of right parabolic
decompositions. Although we do not develop this further, Theorem
\ref{T:bpmodular} implies that an element $w$ of an arbitrary finite Coxeter
group has a complete chain BP decomposition if and only if $L(\mcI(w))$ has a
complete chain of modular elements of a certain form.
\begin{example}\label{Ex:typeAmax}
    Let $W = A_3$ as in Example \ref{Ex:easytypeA}. The maximal element of $W$
    has reduced expression $w_0 = s_1 s_2 s_1 s_3 s_2 s_1$, and Poincare polynomial
    \begin{equation*}
        P_{w_0}(q) = 1 + 3q + 5q^2 + 6q^3 + 5q^4 + 3q^5 + q^6 = [2]_q \cdot [3]_q \cdot [4]_q.
    \end{equation*}
    Let $u_1 = s_1 s_2 s_1$, $v_1 = s_3 s_2 s_1$, and $J_1 = \{s_1,s_2\}$. Then
    $P_{u_1} = 1+2q+2q^2+q^3 = [2]_q \cdot [3]_q$, while ${}^{J_1} P_{v_1}(q) =
    [4]_q$.  Thus $P_{u_1}(q) \cdot {}^{J_1} P_{v_1}(q) = P_{w_0}(q)$, and
    hence $w_0 = u_1 v_1$ is a chain BP decomposition.

    The inversion hyperplane arrangement $\mcI(w_0)$ is cut out by linear forms
    $\alpha_1$, $\alpha_2$, $\alpha_3$, $\alpha_1+\alpha_2$, $\alpha_2 + \alpha_3$,
    and $\alpha_1 + \alpha_2 + \alpha_3$. Let 
    \begin{equation*}
        X_1 = \bigcap \mcI(u_1) = \ker \alpha_1 \cap \ker \alpha_2 \cap \ker (\alpha_1 + \alpha_2)
            = \vspan\{e_3\},
    \end{equation*}
    where $e_i$ is the $i$th standard basis vector in $\C_3$.  According to
    Theorem \ref{T:bpmodular}, $X_1$ is a modular coatom of $L(\mcI(w_0))$. We will verify
    this directly in Example \ref{Ex:typeAmodular}

    Similarly, let $u_2 = s_1$, $v_2 = s_2 s_1$, and $J_2 = \{s_1\}$. Then
    $P_{u_2}(q) = [2]_q$ and ${}^{J_2} P_{v_2}(q) = [3]_q$, so $u_1 = u_2 v_2$
    is also a chain BP decomposition, and $w_0$ has a complete chain BP
    decomposition. Let 
    \begin{equation*} 
        X_2 = \bigcap \mcI(u_2) = \ker \alpha_1 = \vspan\{e_2,e_3\}.
    \end{equation*} 
    Then $X_2$ is a modular coatom of $L(\mcI(u_1))$, and consequently $\C^3 <
    X_2 < X_1 < 0$ is a complete chain of modular flats for $L(\mcI(w_0))$. 
\end{example}
The reader might like to try the analogous computation for the maximal element
of type $B_3$. 
\begin{example}\label{Ex:hardtypeA}
    The rationally smooth elements in Examples \ref{Ex:easytypeA} and \ref{Ex:typeAmax} 
    have complete BP chain decompositions consisting entirely of left BP decompositions.
    However, this is not always possible, even in type $A$. For a
    counterexample, let $W = S_{8}$ be the Weyl group of type $A_7$, and let
    $S = \{s_i : 1 \leq i \leq 7\}$ again be the standard generating set of simple
    transpositions. Set $v_1 = s_7 s_6$, $v_2 = s_2 s_1$, $v_3 = s_4 s_5 s_6$,
    $v_4 = s_4 s_5$, $v_5 = s_4 s_3 s_2$, $v_6 = s_4 s_3$, and $v_7 = s_4$. 
    Now let
    \begin{equation*}
        w = v_2 v_3 \cdots v_7 v_1 = (s_2 s_1) (s_4 s_5 s_6) (s_4 s_5) (s_4 s_3 s_2) (s_4 s_3) (s_4)
            (s_7 s_6).
    \end{equation*}
    The expression above for $w$ is reduced, so $w$ has length $15$. Let $u_i = 
    v_{i+1} \cdots v_7$ for $1 \leq i \leq 7$ (so $u_7 = e$), and set $J_i = S(u_i)$. 
    Then $w$ has a left chain BP decomposition $w = u_1 v_1$ with respect to
    $J_1$, but $u_1$ has no left chain BP decomposition. However, it does have
    a right chain BP decomposition $u_1 = v_2 u_2$ with respect to $J_2$, and
    in fact $u_i = v_{i+1} u_{i+1}$ is a right chain BP decomposition with
    respect to $J_{i+1}$ for all $1 \leq i \leq 6$.

    Since $w$ has a complete chain BP decomposition, it is rationally smooth,
    and the exponents of $w$ are precisely the lengths of the elements $v_i$,
    $1 \leq i \leq 7$. Thus $w$ has exponents $1, 2, 2, 2, 2, 3, 3$, and the
    Poincare polynomial of $w$ is 
    \begin{align*}
        P_w(q) & = [2]_q \cdot [3]_q^4 \cdot [4]_q^2 \\ 
                & = 1 + 7q + 27 q^2 + 73 q^3 + 152 q^4 + 256 q^5 + 358 q^6 + 422 q^7 \\
                &  + 422 q^8 + 358 q^9 + 256 q^{10} + 152 q^{11} + 72 q^{12} + 27 q^{13}
                    + 7 q^{14} + q^{15}. \\
    \end{align*}
    The inversion hyperplane arrangement $\mcI(w)$ is cut out by $15$ hyperplanes in
    $\C^7$. Since $w$ is rationally smooth, Theorem \ref{T:main} and Corollary \ref{C:poincare}
    imply that $\mcI(w)$ is free with 
    \begin{equation*}
        Q_{\mcI(w)}(t) = (1+t)(1+2t)^4 (1+3t)^2.
    \end{equation*}     
\end{example}
    
\begin{example}
    If $w_0$ is the maximal element of type $D_n$, then $\mcI(w_0)$ is not
    supersolvable \cite{HR13}, nor does $w_0$ have a chain BP decomposition
    (see Theorem \ref{T:chainbpexist}).  
\end{example}

\section{Chain Billey-Postnikov decompositions}\label{S:bp}

If $w = uv$ is the left parabolic decomposition of $w$ with respect to $J \subset
S$, then multiplication gives an injective map 
\begin{equation}\label{E:bpmap}
    [e,u] \times \left([e,v] \cap {}^J W\right) \arr [e,w].
\end{equation}
If $x = u_1 v_1$ is the left parabolic decomposition of an element of $[e,w]$ with respect
$J$, then $v_1 \leq v$. However, it is not true that $u_1 \leq u$, even though
$u_1 \leq w$ and $u_1 \in W_J$. 

\begin{lemma}[\cite{BP05}, \cite{OY10}, \cite{RS13a}]\label{L:bp}
    Let $w = uv$ be the left parabolic decomposition of $w$ with respect to $J$.
    The following are equivalent:
    \begin{enumerate}[(a)]
        \item The map in equation (\ref{E:bpmap}) is surjective (hence bijective).
        \item $u$ is the unique maximal element of $[e,w] \cap W_J$. 
        \item $S(v) \cap J \subseteq D_R(u)$. 
        \item $w = uv$ is a left BP decomposition, so $P_w(q) = P_u(q)\cdot {}^J P_v(q)$.
    \end{enumerate}
\end{lemma}
We refer to \cite{OY10} or \cite{RS13b} for more information about BP
decompositions. 

The fundamental theorem for Billey-Postnikov decompositions is:
\begin{thm}[\cite{Ga98} \cite{Bi98} \cite{BP05}]\label{T:bpexist}
    If $w \in W$ is rationally smooth, $\ell(w) \geq 2$, then $w$ has either a
    left or right BP decomposition with respect to $J = S \setminus \{s\}$,
    where $s \in S(w)$ is a leaf of the Dynkin diagram for $W_{S(w)}$. 
\end{thm}
Theorem \ref{T:bpexist} was proved by Gasharov \cite{Ga98} for type $A$, by
Billey \cite{Bi98} in all classical types, and by Billey-Postnikov \cite{BP05}
in the exceptional types. Note that if $w = uv$ is a BP decomposition with
respect to $J$ and $w$ is rationally smooth, then $u$ is rationally smooth, and
${}^J P_v(q)$ is palindromic \cite{BP05}. 

We note four particular details of Theorem \ref{T:bpexist}:
\begin{enumerate}[label=(d.\arabic*)]
    \item\label{detail1} If $W$ is of type $A$, $w \in W$ is smooth, and $s$ is a fixed leaf
        of $W_{S(w)}$, then $w$ has either a left or right BP decomposition
        with respect to $J = S \setminus \{s\}$. (In other words, it is 
        possible to choose the leaf $s$ used in the decomposition, which is not
        necessarily possible in other types.) \cite{Ga98} \cite{Bi98}
    \item\label{detail2} If $W$ is simply-laced and $w \in W$ is rationally smooth with a
        BP decomposition $w = uv$ with respect to $J \setminus \{s\}$, $s$ a
        leaf of $S(w)$, then $v$ is the maximal element of ${}^{J \cap
        S(v)} W_{S(v)}$ (\cite{BP99} \cite{OY10}, see also \cite{RS13b}). 
    \item\label{detail3} If $W$ is of type $A$, $B$, or $G_2$, $s$ is a leaf of $W$,
        $w$ belongs to ${}^J W$, where $J = S \setminus \{s\}$, and ${}^J P_w(q)$
        is palindromic, then $[e,w] \cap {}^J W$ is a chain in Bruhat order (or
        equivalently, ${}^J P_w(q) = [\ell(w)+1]_q$) (\cite{Bi98} \cite{OY10}, see
        also \cite{RS13b}). This is also true in type $F_4$, except when $w$ is
        the maximal element of $W^J$ (\cite{OY10}, see also \cite{RS13b}).
    \item\label{detail4} If $W$ is of type $D$ and $w \in W$ is rationally smooth, then either
        $w$ is the maximal element of $W_{S(w)}$, or $w$ has a chain BP
        decomposition \cite[Proposition 6.3]{Bi98}. Indeed, suppose $w = uv$ is
        a BP decomposition, where $v \in {}^J W$, $u \in W_J$, and $J = S
        \setminus \{s\}$ for some leaf $s$ of $S(w)$. By the second fact above, $v$
        is maximal in ${}^{J \cap S(v)} W_{S(v)}$, so by Lemma \ref{L:bp} either
        $w$ is maximal in $W_{S(w)}$, or $S(v) \subsetneq S(w)$. In the latter
        case $W_{S(v)}$ is of type $A$, and consequently $[e,v] \cap {}^J W$ is
        a chain by fact \ref{detail3} above.  
\end{enumerate}
By facts \ref{detail3} and \ref{detail4} above, if $w \in W$ is rationally smooth and does
not have a chain BP decomposition, then $W$ must have type $D$, $E$, or $F_4$.
In addition, in types $D$ and $F_4$ all non-maximal rationally smooth elements
have a chain BP decomposition. The main object of this section is to
determine which rationally smooth elements in type $E$ do not have a chain BP
decomposition.  We use the following labelling for the Dynkin diagram of $E_8$:
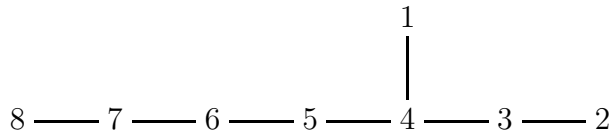
\begin{equation*}
    \vcenter{\xymatrix{ & & & & 1 \ar@{-}[d] & & \\
        8 \ar@{-}[r] & 7 \ar@{-}[r] & 6 \ar@{-}[r] & 5 \ar@{-}[r] & 4 \ar@{-}[r] & 
        3 \ar@{-}[r] & 2 \\ }}
\end{equation*}
We let $S_k = \{s_1,\ldots,s_k\}$, where $s_i$ is the simple reflection
corresponding to the node labelled by $i$, and take the convention that $E_6$
and $E_7$ are embedded inside $E_8$ as $W_{S_6}$ and $W_{S_7}$. Finally, we let
$J_k = S_k \setminus \{s_2\}$. Let $\tilde{u}_k$ be the maximal element of
$W_{J_k}$ and $\tilde{v}_k$ be the maximal element of $W_{S_k}^{J_k}$. Note
that $\tilde{u}_k$ belongs to a parabolic subgroup of type $D$, while
$\tilde{v}_k$ belongs to a parabolic subgroup of type $D$ when $k=5$, and a
parabolic subgroup of type $E$ when $k \geq 6$. The following theorem combines
our type $E$ results with the results for other types mentioned above:
\begin{thm}\label{T:chainbpexist}
    Suppose that $w\in W$ is a rationally smooth element of a finite Weyl group
    $W$, that $\ell(w) \geq 2$, and that $w$ has no chain BP decomposition.
    Then $w$ is one of the following elements:
    \begin{itemize}
        \item The maximal element of a parabolic subgroup of type $D_n$, $n \geq 4$.
        \item The maximal element of a parabolic subgroup of type $E_n$, $n=6,7,8$.
        \item The element $w_{kl} = \tilde{v}_l \tilde{u}_k$, or its inverse
            $w_{kl}^{-1}$, in a parabolic subgroup of type $E_k$, where $5 \leq
            l < k \leq 8$.
        \item The maximal element of $F_4$.
    \end{itemize}
\end{thm}
As discussed above, for classical types Theorem \ref{T:chainbpexist} is proved
in \cite{Ga98} and \cite{Bi98}. We give a proof for type $E$. 
\begin{example}\label{Ex:typeD}
    Let $W = D_5$, with generating set $S = \{s_1,\ldots,s_5\}$, where the
    vertices of the Dynkin diagram are labelled as:   
    
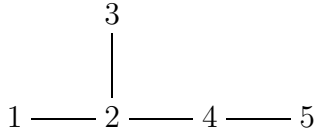
\begin{equation*}
        \xymatrix{              & 3 \ar@{-}[d] &    &      \\
                                            1 \ar@{-}[r] & 2 \ar@{-}[r] & 4 \ar@{-}[r] & 5 }
    \end{equation*}
    Let $w_0$ be the maximal element of $W$, and let $w_1$ be the maximal
    element of $W_J$, where $J = \{s_1,\ldots,s_4\}$. According to Theorem
    \ref{T:chainbpexist}, neither $w_0$ nor $w_1$ have a chain BP decomposition.
    On the other hand, we can find elements like $w = w_1 (s_5 s_4 s_2)$ which
    have a chain BP decomposition, but which do not have a complete chain BP 
    decomposition.
\end{example}

For the proof of Theorem \ref{T:chainbpexist}, we use the following terminology
and lemma from \cite{RS13b}: We say that a subset $T \subset S$ is
\emph{connected} if the Dynkin diagram restricted to $T$ is connected. Note
that if $J = S \setminus \{s\}$ and $w \in W^J$ then $S(w)$ is connected. We
say that two elements $s,t \in S$ are adjacent if they are connected by an edge
in the Dynkin diagram.
\begin{lemma}[\cite{RS13b}, Lemma 5.3]\label{L:blocking}
    Let $w = vu$ be a parabolic decomposition with respect to some $J$, and
    suppose $s \in S$ is adjacent to an element of $S(v)$, but is not contained
    in $S(v)$. Then $s$ does not belong to $D_L(w)$. 
\end{lemma}
\begin{proof}[Proof of Theorem \ref{T:chainbpexist}]
    We now assume that $W$ has type $E$. Note that since $W$ is finite, the theorem could
    be checked on a computer. Instead, we give a complete proof using Lemma
    \ref{L:blocking}. Suppose $w$ is a rationally smooth element of $E_8$, with
    $S(w) = S_k$ for some $k \geq 6$.  Take a BP decomposition of $w$ with
    respect to $J = S_k \setminus \{s\}$, where $s$ is a leaf of $S(w)$. We can
    assume without loss of generality that $w$ has a right BP decomposition of
    this form, so $w = vu$, where $v \in W_{S_k}^{J}$ and $u \in W_{J}$. We
    consider the different possibilities for $s$ separately.

    \emph{Case 1}: $s = s_k$. As noted above, if $[e,v] \cap W^J$ is not a
    chain, then $W_{S(v)}$ must have type $D$ or $E$. By fact \ref{detail2},
    $v$ is maximal in $W_{S(v)}^{S(v) \cap J}$, so if in addition $w$ is not
    maximal of type $E$, then we must have $S(v) = J_k$, and $W_{S(v)}$ must be of
    type $D$. Now $J \cap S(v) = J_{k-1} \subset D_L(u)$ by Lemma \ref{L:bp},
    so we can write $u = u_0 u_1$, where $u_0$ is the maximal element of $W_{J_{k-1}}$.
    Then $v$ is the maximal element of $W_{J_{k}}^{J_{k-1}}$, so $v u_0 =
    \tilde{u}_k$ is the maximal element of $W_{J_k}$. Because $u_0$ is maximal,
    $w = (vu_0) u_1$ is a left BP decomposition with respect to $S(w) \setminus
    \{s_2\}$. If $[e,u_1] \cap {}^{J_k} W$ is not a chain, then $S(u_1) \supset
    S_5$, and $u_1 = \tilde{v}_l^{-1}$ for some $l < k$ by \ref{detail2}. So $w
    = \tilde{u}_k \tilde{v}_l^{-1} = w_{kl}^{-1}$.  

    \emph{Case 2}: $s = s_1$. If $[e,v] \cap W^J$ is not a chain, and $w$ is not
    maximal, then we must have $S(v) \supset J_5$, and $s_2$ and $s_k$
    cannot both belong to $S(v)$. Suppose $S(v)$ contains one of $\{s_2,s_k\}$,
    and let $t$ be the element which is not contained in $S(v)$. 
    Since $W_{S(u)}$ has type $A$, the element $u$ has a BP decomposition with
    respect to $J' = J \setminus \{t\}$ by fact \ref{detail1}. If this is a left BP decomposition
    $u = u_0 u_1$, $u_0 \in W_{J'}$, $u_1 \in {}^{J'} W_J$, then $w = (vu_0) u_1$
    is a left chain BP decomposition. Otherwise, we have a right BP decomposition
    $u = u_1 u_0$, $u_0 \in W_{J'}$, $u_1 \in W^{J'}_J$. Then $W_{S(u_1)}$ has
    type $A$ and $[e,u_1] \cap W_{J}^{J'}$ is a chain by fact \ref{detail3}. If $S(u_1)$ and $S(v)$
    pairwise commute, then $w = u_1 (v u_0)$ is a chain BP decomposition.
    Otherwise, an element of $S(v)$ must be adjacent to or contained
    in $S(u_1)$. But since $S(v) \setminus \{s_1\} $ is connected and is
    contained in $D_L(u)$ by Lemma \ref{L:bp}, we can conclude that $S(v)
    \setminus \{s_1\} \subset S(u_1)$ by Lemma \ref{L:blocking}.  Consequently
    $S(u_1)$ contains $\{s_2,s_k\}$, so $u_0$ is the maximal element of
    $W_{J'}$, and $u$ is the maximal element of $W_{J}$ by \ref{detail2}. Therefore $u$ has a
    left BP decomposition with respect to $J'$. 

    This leaves the possibility that neither $s_2$ or $s_k$ lies in $S(v)$.
    By \ref{detail1}, we can take a BP decomposition of $u$ with respect to $J'
    = J \setminus \{s_k\}$.  If we get a left decomposition, or a right
    decomposition $u = u_1 u_0$
    where $S(u_1)$ pairwise commutes with $S(v)$, then we get a chain
    decomposition of $w$ as above. Suppose we get a right BP decomposition $u =
    u_1 u_0$, $u_1 \in W_{J}^{J'}$, $u_0 \in W_{J'}$, where $S(v)$ and $S(u_1)$
    do not pairwise commute. By Lemma \ref{L:blocking}, $S(u_1) \supset S(v)
    \setminus \{s_1\}$. If $S(u_1) = J$, then as above $u$ is the maximal
    element of $W_{J}$, so we can take a left BP decomposition of $u$ to get a
    chain decomposition of $w$. Otherwise $S(u_1) = J \setminus \{s_2\}$ since
    $J_5 \subset S(v) \setminus \{s_1\}$ and $S(u_1)$ is connected. Since 
    $S(u_1) \cap J' = J' \setminus \{s_2\} \subset D_L(u_0)$ by Lemma
    \ref{L:bp}, we can write $u_0 = u_0' u_0''$, where $u_0'$ is the maximal
    element of $W_{J' \setminus \{s_2\}}$ and $u_0'' \in {}^{J' \setminus
    \{s_2\}} W_{J'}$. Then $w = (v u_1 u_0') u_0''$ is a chain BP
    decomposition by \ref{detail3}. 

    \emph{Case 3}: $s = s_2$. In this case, $v \in W^{J_k}_{S_k}$. If $[e,v]
    \cap W^{J_k}$ is not a chain and $w$ is not maximal, then $S(v) = S_l$ for
    some $k > l \geq 5$, and $v = \tilde{v}_l$ by \ref{detail2}. If $u$ has a left BP
    decomposition $u = u_0 u_1$ with respect to $J_{k-1} = S(u) \setminus
    \{s_k\}$, then either $W_{S(u_1)}$ is of type $A$, in which case $w =
    (vu_0) u_1$ is a chain BP decomposition by \ref{detail3}, or $S(u_1) = J_k$, in which case
    $u = \tilde{u}_k$ by \ref{detail2}, so that $w = w_{kl}$. Similarly if $u$ has a right BP
    decomposition $u = u_1 u_0$ with respect to $J_{k-1}$ and $S(u_1)$ pairwise
    commutes with $S(v)$, then $w = u_1(vu_0)$ is a chain BP decomposition
    (since $S(u_1)$ cannot meet $S(v)$, $W_{S(u_1)}$ is of type $A$ and we can
    apply \ref{detail3}). If $S(u_1)$ does not pairwise commute with $S(v)$
    then $S(u_1) \supset J_l$, so $S(u_1) = J_k$ by Lemma \ref{L:blocking}, and
    again $u = \tilde{u}_k$ by \ref{detail2}.

    If $u$ does not have a BP decomposition with respect to $J_{k-1}$, then $u$
    has a BP decomposition with respect to $J' = J_k \setminus \{t\}$, where $t
    \in \{s_1,s_3\}$. Suppose $u$ has a right BP decomposition with respect to
    $J'$, so $u = u_1 u_0$, $u_1 \in W^{J'}_{J_k}$, $u_0 \in
    W_{J'}$. Then $S(u_1)$ and $S(v)$ both contain $t$, so by
    Lemma \ref{L:blocking}, $S(u_1) \supset S(v) \setminus \{s_2\} = J_l$.
    If $s_k \in S(u_1)$ then $S(u) = J_k$ and $u = \tilde{u}_k$ by
    \ref{detail2}. Otherwise $W_{S(u_0)}$ is of type $A$, so by \ref{detail1} we can take a
    decomposition of $u_0$ with respect to $K = J' \setminus \{s_k\}$. If we
    get a left BP decomposition of $u_0$, or a right BP decomposition $u_0 =
    u_0'' u_0'$, $u_0'' \in {}^K W_{J'}$, $u_0' \in W_K$, where $S(u_0'')$
    pairwise commutes with $S(u_1)$, we get a chain BP decomposition of $w$ by \ref{detail3}.
    Otherwise by Lemma \ref{L:blocking} we must have a right BP decomposition
    $u_0 = u_0'' u_0'$ where $S(u_0'') \supset S(u_1) \setminus \{t\}$. Since
    $s_k \in S(u_0'')$, fact \ref{detail2} implies that $u_0$ is the maximal
    element of $W_{J'}$, and hence has a left BP decomposition with respect to $K$. 

    Finally, suppose that $u$ has a left BP decomposition $u = u_0 u_1$ where
    $u_1 \in {}^{J'} W_{J_k}$, $u_0 \in W_{J'}$. If $W_{S(u_1)}$ is of type $A$,
    then $w$ has a chain BP decomposition $w = (v u_0) u_1$ by \ref{detail3}. If $W_{S(u_1)}$ is
    not of type $A$ then $S(u_1)$ must contain $J_5$. If $s_k \in S(u_1)$ then
    $u = \tilde{u}_k$ by \ref{detail2}. Suppose $s_k \not\in S(u_1)$. Since $S(u_0)$ is of type
    $A$, by \ref{detail1} we can take a BP decomposition of $u_0$ with respect to $K$.  If $u_0
    = u_0' u_0''$ where $u_0'' \in W_{J'}^{K}$ and $u_0' \in W_{K}$, and
    $S(u_0'')$ pairwise commutes with $S(u_1)$ then $w = (v u_0' u_1) u_0''$
    is a chain BP decomposition by \ref{detail3}. If $S(u_0'')$ does not pairwise commute
    with $S(u_1)$ then by Lemma \ref{L:blocking}, $S(u_0'') \supset S(u_1)
    \setminus \{t\}$, so $u_0$ is the maximal element of $W_{J'}$ by
    \ref{detail2}, and consequently $J' \subseteq D_L(u)$. Since $t \in S(v)$,
    we must have $t \in D_L(u)$ and hence
    $u = \tilde{u}_k$.  Similarly if $u_0 = u_0'' u_0'$ where $u_0'' \in
    W_{J'}^K$ and $u_0' \in W_{K}$, then $t \not\in S(u_0'')$. Since $s_k \not\in
    S(u_1)$, $u = u_0'' (u_0' u_1)$ is a parabolic decomposition with respect to
    $J_{k-1}$, and it follows from Lemma \ref{L:blocking} that $S(u_0'')$ must
    pairwise commute with $S(v)$, so $w = u_0'' (v u_0' u_1)$ is a chain BP
    decomposition by \ref{detail3}.
\end{proof}

\subsection{Proof of Theorem \ref{T:exponents}}
We can now give a short account of the proof of Theorem
\ref{T:exponents}. If $w$ has a chain BP decomposition $w = vu$ or $w = uv$,
then $P_w(q) = [\ell(v)+1]_q P_u(q)$, and since $u$ is also rationally smooth,
we can proceed by induction. Thus we only need to check Theorem
\ref{T:exponents} for elements which do not have a chain BP decomposition.
If $w$ is the maximal element in $W_{S(w)}$, then as noted in the introduction
the conclusion of Theorem \ref{T:exponents} is well-known.  This leaves the
elements $w_{kl}$ and $w_{kl}^{-1}$ from Theorem \ref{T:chainbpexist}. The
inverse map gives an order-isomorphism $[e,w] \iso [e,w^{-1}]$, so $P_w(q)
= P_{w^{-1}}(q)$. Thus we only need to check the theorem for the elements
$w_{kl}$, $5 \leq l < k \leq 8$. Let $\tilde{w}_{l}$ be the maximal element of
$W_{S_l}$.  Then
\begin{equation*}
    P_{w_{kl}}(q) = P_{\tilde{v}_l}^{J_l}(q)\  P_{\tilde{u}_k}(q) = 
        P_{\tilde{w}_l}(q)\  P_{\tilde{u}_k}(q)\  P_{\tilde{u}_{l}}^{-1}(q),
\end{equation*}
where the last equality uses the fact that $\tilde{w}_l = \tilde{v}_l
\tilde{u}_{l}$. The elements $\tilde{u}_k$ are
maximal of type $D$, and the elements $\tilde{w}_l$ are maximal of
types $D$ ($l=5$) and $E$ ($l>5$). We can check that the exponents
of $\tilde{u}_l$ are, counting with multiplicity, contained in the
union of the exponents of $\tilde{w}_l$ and $\tilde{u}_k$, so that
$P_{w_{kl}}(q)$ is a product of $q$-integers. The exponents of the elements
$w_{kl}$ are given in Table \ref{TBL:exceptexp}.

\begin{table}[h]
    \begin{tabular}{ll}
        $w\quad \quad$ & Exponents of $w$ (zeroes omitted) \\
        $w_{65}$ & 1, 4, 4, 5, 7, 7 \\
        $w_{75}$ & 1, 4, 5, 5, 7, 7, 9 \\
        $w_{85}$ & 1, 4, 5, 6, 7, 7, 9, 11 \\
        $w_{76}$ & 1, 5, 5, 7, 8, 9, 11 \\
        $w_{86}$ & 1, 5, 6, 7, 8, 9, 11, 11 \\
        $w_{87}$ & 1, 6, 7, 9, 11, 11, 13, 17 \\
    \end{tabular}
    \caption{Table of exponents for elements $w_{kl}$.}
    \label{TBL:exceptexp}
\end{table}

It is not necessary to use Theorem \ref{T:chainbpexist} to prove Theorem
\ref{T:exponents}. When the Schubert variety $X(w)$ is smooth, Theorem
\ref{T:exponents} follows from a result of Akyildiz-Carrell \cite{AC12}, which
states that the exponents can be calculated from the torus weights $\Omega_w$
of the tangent space $T_e X(w)$ to $X(w)$ at the identity, analogously to how
the exponents of $W$ can be determined from the heights of the positive roots
of $R$. In the simply-laced types, all rationally smooth Schubert varieties are
smooth by a theorem of Peterson \cite{CK03}, so the Akyildiz-Carrell theorem
covers Theorem \ref{T:exponents}. The non-simply-laced types are covered 
by \cite{Bi98} (classical types) and \cite{OY10} (type $F_4$).

\section{The HLSS condition and $\nbc$-sets}\label{S:hultman}

In this section we give some background on the HLSS condition which is
necessary in the subsequent sections.  Given a reduced expression $s_1 \cdots
s_{\ell(w)}$ for an element $w \in W$, we can order the inversion set $I(w)$ by
$\beta_1 < \cdots < \beta_{\ell(w)}$, where $\beta_i = s_1 \cdots s_{i-1}
\alpha_{s_i}$. A total order on $I(w)$ constructed in this way is called a
\emph{convex order} (this is also called a \emph{reflection order}, see e.g.
\cite[Section 5.2]{BB05}). Let $2^{I(w)}$ denote the power set of $I(w)$.
Given a convex order, we can define a surjective map
\begin{equation*}
    \phi\ :\ 2^{I(w)} \arr [e,w] \ : \ \{\beta_1,\ldots,\beta_k\} \mapsto 
        t_{\beta_1} \cdots t_{\beta_k} w \text{ when } \beta_1 < \cdots < \beta_k.
\end{equation*}
\begin{thm}[Hultman-Linusson-Shareshian-Sj\"ostrand \cite{HLSS09}]\label{T:HLSS}
    Choose a convex order for $I(w)$, and let $\nbc(I(w))$ denote the set of
    $\nbc$-sets of $I(w)$ with respect to the chosen order. Then the
    restriction of $\phi$ to $\nbc(I(w))$ is injective.
\end{thm}
In particular the number of $\nbc$-sets of $I(w)$ is less than the size of the
Bruhat interval $[e,w]$. The restriction of $\phi$ to $\nbc(I(w))$ will be
surjective if and only if the number of $\nbc$-sets is equal to the size of the
Bruhat interval. Since the size of $\nbc(I(w))$ is independent of the choice of
convex order, if there is some convex order for which the restriction of $\phi$
to $\nbc(I(w))$ is surjective, then this happens for all convex orders. 

A theorem of Hultman-Linusson-Shareshian-Sj\"ostrand (type $A$) and Hultman
(all finite Coxeter groups) characterizes when the restriction is surjective.
Recall that $\ell'(w)$ is the absolute length of $w$, and $\al(u,w)$ is the
distance from $u$ to $w$ in the directed Bruhat graph. The distance from $u$ to
$w$ in the undirected Bruhat graph is simply $\ell'(u w^{-1})$, so $\al(u,w)
\geq \ell'(u w^{-1})$. 
\begin{thm}[\cite{HLSS09} \cite{Hu11}]\label{T:hultman1}
    The restriction of $\phi$ to $\nbc(I(w))$ is surjective (and hence
    bijective) if and only if $\al(u,w) = \ell'(u w^{-1})$ for all $u \leq w$.
\end{thm} 

\begin{defn}
    We say that $w$ satisfies the \emph{HLSS condition} if $\al(u,w) =
    \ell'(u w^{-1})$ for all $u \leq w$, or equivalently if the restriction of
    $\phi$ to $\nbc(I(w))$ is surjective.
\end{defn}
The HLSS condition is weaker than being rationally smooth:
\begin{thm}[\cite{Hu11}]\label{T:hultman2}
    The HLSS condition is satisfied if $w$ is rationally smooth.
\end{thm}

The following theorem of Carter, used by Hultman in the proof of Theorem \ref{T:hultman1},
implies that if $\{\beta_{j_1},\ldots,\beta_{j_k}\} \subset I(w)$ is linearly independent, then
$\ell'\left(\phi(\beta_{j_1}, \ldots, \beta_{j_k}) w^{-1} \right) = k$.
\begin{thm}[\cite{Ca72}]\label{T:carter}
    Suppose $w = t_{\beta_1} \cdots t_{\beta_m}$, where $t_{\beta_i}$ refers to 
    reflection through the root $\beta_i$. Then 
    \begin{enumerate}[(a)]
        \item $\ell'(w) = m$ if and only if $\{\beta_1,\ldots,\beta_m\}$ is 
            linearly independent.
        \item The fixed point space of $w$ on $V$ contains the orthogonal
            complement of $\vspan\{\beta_1,\ldots,\beta_m\}$. 
        \item The codimension of the fixed point space of $w$ is equal to
            $\ell'(w)$, so if $\ell'(w) = m$ then the fixed point space of $w$
            is equal to the orthogonal complement of $\vspan\{\beta_1,\ldots,\beta_m\}$. 
    \end{enumerate}
\end{thm}
In particular, Theorem \ref{T:carter} (a) applies if $\{\beta_1,\ldots,\beta_m\}$ is
an $\nbc$-set, since $\nbc$-sets are always linearly independent.  We will use
Theorem \ref{T:carter} later in Section \ref{S:flattening}.  To finish the
section, note that Equation (\ref{E:HLSS}) from the introduction follows
immediately from the results outlined in this section:
\begin{cor}\cite{HLSS09}
    If $w$ satisfies the HLSS condition then
    \begin{equation*}
        Q_{\mcI(w)}(q) = \sum_{u \in [e,w]} q^{\ell'(uw^{-1})} = \sum_{u \in [e,w]} q^{\al(u,w)}.
    \end{equation*}
\end{cor}

\section{Chain BP decompositions and modular flats}\label{S:modular}

In this section, we prove Theorem \ref{T:bpmodular}, and use it to prove
one direction of Theorem \ref{T:main}.  To determine when $X$ is a coatom of
$\mcI(w)$ we use the following lemma.  Recall that $V$ is the ambient Euclidean
space containing $R$. Given $J \subset S$, let $V_J \subset V$ be the subspace
spanned by $\{\alpha_s : s \in J\}$, and let $R_J = R \cap V_J$ be the root
system for $W_J$. 
\begin{lemma}\label{L:invspan}
    The linear span of the inversion set $I(w)$ in $V$ is $V_{S(w)}$, and the
    center of the inversion hyperplane arrangement $\mcI(w)$ is the orthogonal
    complement of $V_{S(w)}$.
\end{lemma}
\begin{proof}
    The proof is by induction on $\ell(w)$. The lemma is clearly true for
    the identity. Given $w \neq e$, choose $s \in D_L(w)$. Then $I(w) =
    \{\alpha_s\} \cup s I(sw)$, so the span of $I(w)$ is 
    \begin{equation*}
        \R \alpha_s + \vspan I(sw) = \R \alpha_s + V_{S(sw)} = V_{S(w)}.
    \end{equation*}
    The corresponding statement for $\mcI(w)$ follows immediately.
\end{proof}
Lemma \ref{L:invspan} implies that the rank of $\mcI(w)$ is the size of the
support set $S(w)$.  Recall that a coatom of $L(\mcI(w))$ is an element of
$L(\mcI(w))$ of rank $|S(w)|-1$, or in other words is an $(l -|S(w)| +
1)$-dimensional subspace of $V$ which can be written as an intersection of
hyperplanes in $\mcI(w)$. The flat $X = \bigcap \mcI(u)$ has rank $|S(u)|$,
and hence $X$ will be a coatom if and only if $|S(u)| = |S(w)|-1$.

To prove Theorem \ref{T:bpmodular} we also need the following standard
characterization of modular coatoms (see for instance \cite{RR13}).
\begin{lemma}\label{L:modular}
    Let $\mcA$ be an arrangement, and let $\{\alpha_H\}$ be a set of defining
    forms for $\mcA$. Let $X \in L(\mcA)$ be a coatom. Then $X$ is modular if
    and only if for every distinct pair $H_1, H_2 \not\in \mcA_X$, there is
    $H_3 \in \mcA_X$ such that $\alpha_{H_1}, \alpha_{H_2}, \alpha_{H_3}$ are
    linearly dependent.
\end{lemma}
\begin{proof}
    Proposition 4.2 of \cite{RR13} states that $X$ is modular if and only if 
    for every distinct $H_1, H_2 \not\in \mcA_X$, there is $H_3 \in \mcA$ such
    that $H_1 \cap H_2 + X = H_3$. This means that $H_3$ must contain both $X$
    and $H_1 \cap H_2$, and the latter condition is equivalent to the condition
    that $\alpha_{H_3}$ be in the span of $\alpha_{H_1}$ and $\alpha_{H_2}$. 
\end{proof}
\begin{example}\label{Ex:typeAmodular}
    In Example \ref{Ex:typeAmax}, it is claimed that $X_1 = \vspan\{e_1\}$
    is a modular coatom of $L(\mcI(w_0))$, where $w_0$ is the maximal element of
    $A_3$. Since $X_1$ has dimension $1$ and $\mcI(w_0) \subset \C^3$ has 
    trivial center, $X_1$ is indeed a coatom. Now 
    \begin{equation*}
        \mcI(w_0)_{X_1} = \{\ker \alpha_2, \ker \alpha_3, \ker (\alpha_2 + \alpha_3)\},
    \end{equation*}
    while 
    \begin{equation*}
        \mcI(w_0) \setminus \mcI(w_0)_{X_1} = \{\ker \alpha_1, \ker (\alpha_1+\alpha_2),
            \ker (\alpha_1+\alpha_2+\alpha_3)\}.
    \end{equation*}
    To show that $X_1$ is modular using Lemma \ref{L:modular}, we need to show that
    the span of every pair of elements from $\alpha_1$, $\alpha_1 + \alpha_2$, and $\alpha_1
    + \alpha_2 + \alpha_3$ contains one of $\alpha_2, \alpha_3$, or $\alpha_2 + \alpha_3$.
    But this is clearly true: for instance, $(\alpha_1 + \alpha_2 + \alpha_3) - \alpha_1
    = (\alpha_2 + \alpha_3)$. 
\end{example}
\begin{proof}[Proof of Theorem \ref{T:bpmodular}]
    Recall that $X = \bigcap \mcI(u)$, where $w = uv$ is a left parabolic 
    decomposition of $w$ with respect to $J$.  If $w = uv$ is a chain BP
    decomposition, then $|S(w) \cap J| = |S(w)|-1$, and $S(v) \cap J \subset
    S(u)$ so $|S(u)| = |S(w)|-1$.  Hence we can assume throughout that $X$ is a
    coatom.
    
    Recall that $I(w)$ is the union of $I(u)$ and $uI(v)$. The set $I(u)$ is equal
    to $I(w) \cap V_J$, and $w \in {}^J W$ if and only if $I(w) \cap V_J$ is
    empty. The hyperplanes of $\mcI(w)$ which do not contain $X$ correspond
    precisely to the roots in $u I(v)$. By Lemma \ref{L:modular}, $X$ is
    modular if and only if for all $\alpha, \beta \in uI(v)$, there is $\gamma
    \in I(u)$ such that $\alpha, \beta, \gamma$ are linearly dependent. The
    inversion set $I(u^{-1}) = -u^{-1}I(u)$, so applying $-u^{-1}$ we see that
    $X$ is modular if and only if for all $\alpha, \beta \in I(v)$ there is
    $\gamma \in I(u^{-1})$ such that $\alpha, \beta, \gamma$ are linearly
    dependent. 

    Now suppose that $w = uv$ is a chain BP decomposition. We first assume that
    $u$ satisfies the HLSS condition. Take any total order on $I(w)$ in which
    the elements of $I(u)$ come after the elements of $uI(v)$. Every element
    $\alpha \in u I(v)$ is independent from the span of $\mcI(u)$. Thus
    if $\{\gamma_1,\dots, \gamma_k\}$ is an $\nbc$-set for $\mcI(u)$, then
    $\{\alpha,\gamma_1,\ldots, \gamma_k\}$ is an $\nbc$-set for $\mcI(w)$. So
    \begin{equation*}
        |\nbc(\mcI(w))| \geq (1+\ell(v)) \cdot |\nbc(\mcI(u))|.
    \end{equation*}
    Since $u$ satisfies the HLSS condition, $|\nbc(\mcI(u))| = |[e,u]|$, while
    $|\nbc(\mcI(w))| \leq |[e,w]|$. But $w = uv$ is a chain BP decomposition, so 
    \begin{equation*}
        |[e,w]| = P_w(1) = P_u(1) \cdot {}^J P_v(1) = (1 + \ell(v)) \cdot |[e,u]| 
            = (1+\ell(v)) \cdot |\nbc(\mcI(u))|.
    \end{equation*}
    We conclude that all $\nbc$-sets of $\mcI(w)$ in the chosen order are
    either $\nbc$-sets of $\mcI(u)$, or are of the form
    $\{\alpha,\gamma_1,\ldots,\gamma_k\}$ for $\alpha \in u I(v)$, and
    $\{\gamma_1,\ldots,\gamma_k\}$ an $\nbc$-set of $\mcI(u)$. 

    In particular, if $\alpha, \beta \in u I(v)$ then $\{\alpha, \beta\}$ is
    not an $\nbc$-set in any order on $I(w)$ in which the elements of $I(u)$
    come after the elements of $u I(v)$. Ordering $I(w)$ in this way so that
    $\beta$ is the last element of $u I(v)$, we see that there is some $\gamma
    \in I(u)$ such that $\alpha, \beta, \gamma$ are linearly dependent. By
    Lemma \ref{L:modular}, $X$ is modular.

    If $u$ does not satisfy the HLSS condition, we can replace $w$ with $w' =
    u_0 v$, where $u_0$ is the maximal element of $W_{J \cap S(v)}$. Since
    $u_0$ is rationally smooth, it satisfies the HLSS condition, and by the
    above argument, the coatom $X' = \bigcap \mcI(u_0)$ is modular.  This means
    that for all $\alpha, \beta \in I(v)$, there is $\gamma \in I(u_0^{-1})$
    such that $\alpha, \beta, \gamma$ are linearly dependent. But since $w =
    uv$ is a BP decomposition, $u$ has a reduced decomposition $u = u_1 u_0$ by
    Lemma \ref{L:bp} (c), and hence $I(u_0^{-1}) \subset I(u^{-1})$. Thus
    $X$ is modular.

    Now suppose that $X$ is a modular coatom. We want to show that $w = uv$
    is a chain BP decomposition. We start with the following claim: if $v$ has
    a reduced decomposition $v = v_0 s t v_1$, where $s,t \in S$, and $t
    \not\in D_R(v_0)$, then $v_0 t \not\in {}^J W$, and in fact $v_0 t = 
    t' v_0$, where $t' \in D_R(u)$. Indeed, let $\alpha = v_0 \alpha_s$ and
    $\beta = v_0 s \alpha_t$.  Since $\alpha, \beta \in I(v)$ and $X$ is
    modular, there is $\gamma \in I(u^{-1})$ such that $\alpha, \beta, \gamma$
    are linearly dependent. Now
    \begin{equation*}\label{E:reducedmodular}
        \beta = v_0 \left(\alpha_t - 2\frac{(\alpha_s, \alpha_t)}{(\alpha_s,\alpha_s)} \alpha_s\right)
            = v_0 \alpha_t - 2 \frac{(\alpha_s, \alpha_t)}{(\alpha_s,\alpha_s)} \alpha.
    \end{equation*}
    Hence the span of $\alpha$ and $\beta$ is the same as the span of $v_0
    \alpha_s$ and $v_0 \alpha_t$, and we can find $a,b \in \R$ such that
    \begin{equation}\label{E:reducedmodular2}
        a v_0 \alpha_s + b v_0 \alpha_t = \gamma.
    \end{equation}
    Since $\gamma \in V_J$ and $v_0 \in {}^J W$, the root $\gamma$ cannot be in
    $I(v_0)$, and hence $a \alpha_s + b \alpha_t = v_0^{-1} \gamma \in R^+$.
    It follows that $a, b \geq 0$. Let $S(v) \setminus S(u) = \{r\}$. Now
    $v_0 \alpha_s$ is in $I(v)$, and in particular is positive and does not
    belong to $V_J$. It follows that if we write $\alpha$ as
    a linear combination of simple roots, the coefficient of $\alpha_r$ will be
    positive. Since $t \not\in D_R(v_0)$, $v_0 t$ is reduced and $v_0 \alpha_t
    \in I(v)$ is also positive. Hence the coefficient of $\alpha_r$ in $v_0
    \alpha_t$ is non-negative. But $a v_0 \alpha_s + b v_0 \alpha_t = \gamma$
    belongs to $V_J$, so we must have $a=0$, $b=1$, and $v_0 \alpha_t = \gamma
    \in I(u^{-1})$. Since $v_0 \alpha_t \in I(v_0 t) \cap V_J$, the element
    $v_0 t$ does not belong to ${}^J W$, and $v_0 t$ has a non-trivial
    parabolic decomposition $v_0 t = v_0' v_0''$, $v_0' \in W_J$, $v_0'' \in
    {}^J W$. Since $v_0 \leq v_0 t$, we have $v_0 \leq v_0''$, and by comparing
    lengths we see that $v_0 = v_0''$, while $v_0' = t'$ for some $t' \in J$.
    The simple reflection $\alpha_{t'}$ is the unique element of $I(v_0 t) \cap
    V_J$, so $\alpha_{t'} = v_0 \alpha_t \in I(u^{-1})$, and consequently $t'
    \in D_L(u^{-1}) = D_R(u)$. This finishes the proof of the claim.

    Next we show that $w = uv$ is a BP decomposition. Indeed, if $t \in S(v)
    \cap J$, we can find a reduced decomposition $v = v_0 s t v_1$, $s \in S$,
    where $t \not\in S(v_0 s)$. This latter fact implies that $v_0 t$ is
    reduced, so by the above claim $v_0 t = t' v_0$ where $t' \in D_R(u)$. 
    But $t \in S(t' v_0)$ and $t \not\in S(v_0)$, so $t = t'$. Hence $w = uv$
    is a BP decomposition.

    It remains to show that $[e,v] \cap {}^J W$ is a chain. We use induction on
    $\ell(v)$. If $\ell(v) = 1$ the claim is obvious, so suppose $\ell(v)>1$.
    Let $v = s_1 \cdots s_k$ be a reduced factorization of $v$ into simple
    reflections, and let $\hat{v}_i = s_1 \cdots s_i$. We want to show that
    $[e,v] \cap {}^J W = \{\hat{v}_i : i=1,\ldots,k\}$. Suppose that $v' =
    s_1 \cdots s_{i-1} s_{i+1} \cdots s_k$ is an element of $[e,v] \cap {}^J W$
    of length $\ell(v)-1$. Since $|S(v) \setminus J| = 1$, we must have $i >
    1$. If $i < k$ then we get a reduced decomposition $v = \hat{v}_{i-1} s_i
    s_{i+1} v_1$, where $v_1 = s_{i+2} \cdots s_k$.  Since $v_0 s_{i+1} v_1$ is
    reduced, $s_{i+1} \not\in D_R(v_0)$. Hence by the above claim, $v_0 s_{i+1}
    \not\in {}^J W$. But this contradicts the fact that $v' \in {}^J W$. Hence
    the only element of $[e,v] \cap {}^J W$ of length $\ell(v)-1$ is
    $\hat{v}_{k-1}$. The flat $X$ remains a modular coatom of $\mcI(u
    \hat{v}_{k-1})$, so by induction $[e,\hat{v}_{k-1}] \cap {}^J W =
    \{\hat{v}_i : i=1,\ldots, k-1\}$. If $v'$ is a general element of $[e,v]
    \cap {}^J W$, $v' \neq v$, then by Corollary 3.8 of \cite{De77} there is a
    chain $v' = x_0 < x_1 < \cdots < x_m = v$, where $x_i \in {}^J W$ and
    $\ell(x_i) = \ell(x_{i-1})+1$ for $i=1,\ldots,m$. As just shown, 
    we must have $x_{m-1} = \hat{v}_{k-1}$, and hence $v' \in
    [e,\hat{v}_{k-1}] \cap {}^J W$.  We conclude that $v' = \hat{v}_i$ for
    some $i$.
\end{proof}
\begin{rmk}
    The proof of Theorem \ref{T:bpmodular} implies that $[e,v] \cap {}^J W$
    is a chain for $v \in {}^J W$ if and only if the span of every pair
    of elements $\alpha, \beta \in I(v)$ contains an element of $R_J$. Thus,
    checking whether or not $[e,v] \cap {}^J W$ is a chain reduces to a linear
    condition on $I(v)$.
\end{rmk}
\begin{rmk}\label{R:supersolv}
    Recall that an arrangement $\mcA$ is supersolvable if $L(\mcA)$ has a
    complete chain of modular elements. Equivalently, $\mcA$ is supersolvable
    if and only if $L(\mcA)$ has a modular coatom $X$, and the localization
    $\mcA_X$ is supersolvable.  Hence Theorem \ref{T:bpmodular} implies
    immediately that $\mcI(w)$ is supersolvable when $w$ has a complete chain
    BP decomposition.  
\end{rmk}
The proof of Theorem \ref{T:bpmodular} given above is quite lengthy. For the
purposes of proving Theorems \ref{T:main} and \ref{T:exponents}, we can assume
that $u$ and $w$ satisfy the HLSS condition. With these assumptions, there is
a much shorter proof of Theorem \ref{T:bpmodular} as follows: Suppose $X$ is
modular. Because $u$ and $w$ satisfy the HLSS condition, 
\begin{equation*}
    |[e,w]| = |\nbc(\mcI(w))| = |\nbc(\mcI(u))| \cdot \left(|\mcA \setminus \mcA_X|+1\right)
        = |[e,u]| \cdot (\ell(v)+1).
\end{equation*} 
On the other hand, $|[e,w]| \geq |[e,u]| \cdot |[e,v] \cap {}^J W|$, and
$|[e,v] \cap {}^J W| \geq \ell(v)+1$. So $|[e,v] \cap {}^J W| = \ell(v)+1$, and
consequently ${}^J P_v(q) = [\ell(v)+1]_q$. Furthermore, the multiplication map
$[e,u] \times \left([e,v] \cap {}^J W\right) \arr [e,w]$ will be surjective, so
$w = uv$ is a BP decomposition.

We now use Theorem \ref{T:bpmodular} to prove one direction of Theorem
\ref{T:main}:
\begin{thm}\label{T:main1}
    If $w \in W$ is rationally smooth then $\mcI(w)$ is inductively free, and
    the coexponents of $\mcI(w)$ are equal to the exponents of $w$.
\end{thm}
Our motivation for studying modular coatoms is the following lemma. Although
this lemma is likely well-known, we give the proof for completeness.
\begin{lemma}\label{L:modularfree}
    If $X$ is a modular coatom for $\mcA$, and $\mcA_X$ is inductively free
    with coexponents $0,m_1,\ldots,m_{l-1}$, then $\mcA$ is inductively free
    with coexponents $m_1,\ldots,m_{l-1}$, $m_l = |\mcA| - |\mcA_X|$. 
\end{lemma}
Note that $0$ is a coexponent of $\mcA_X$ of multiplicity at least one, since
the center of $\mcA_X$ is non-trivial.
\begin{proof}[Proof of Lemma \ref{L:modularfree}]
    The proof is by induction on the size of $\mcA - \mcA_X$.  If $H \in \mcA
    \setminus \mcA_X$, then $\alpha_H$ is independent from $\{\alpha_{H'} : H'
    \in \mcA_X\}$. By Lemma \ref{L:modular}, if $X$ is modular then $\mcA^H
    \iso \mcA_X / X$, since the restriction of any hyperplane in $\mcA
    \setminus \mcA_X$ to $H$ agrees with the restriction of a hyperplane in
    $\mcA_X$. Using Lemma \ref{L:modular} again, we see that $X$ is modular
    in $\mcA \setminus H$, so the lemma follows from the addition
    theorem.
\end{proof}

\begin{proof}[Proof of Theorem \ref{T:main1}]
    The proof is by induction on $|S(w)|$. Clearly the proposition is true if
    $|S(w)| \leq 1$. Suppose $w$ has a chain BP decomposition. The element $w$
    is rationally smooth if and only if $w^{-1}$ is rationally smooth.  Since
    $I(w^{-1}) = -w^{-1} I(w)$, the arrangements $\mcI(w)$ and $\mcI(w^{-1})$ are
    linearly equivalent. Thus the proposition holds for $w$ if and only if it
    holds for $w^{-1}$, and we can assume without loss of
    generality that $w$ has a right chain BP decomposition $w = uv$.  Then $u$
    is also rationally smooth, and $P_w(q) = [\ell(v)+1]_q P_u(q)$, so if
    the exponents of $u$ are $0,m_1,\ldots,m_{l-1}$, then the exponents of $w$ are
    $m_1,\ldots,m_{l-1}, \ell(v)$. The coatom $X = \bigcap \mcI(u)$ is modular
    by Theorem \ref{T:bpmodular}.  The arrangement $\mcI(w)_{X}$ is simply
    $\mcI(u)$, which by induction is free with coexponents equal to
    $0,m_1,\ldots,m_{l-1}$. By Lemma \ref{L:modularfree}, the arrangement
    $\mcI(w)$ is free with coexponents $m_1,\ldots,m_{l-1},\ell(v)$, since
    $|\mcI(w) \setminus \mcI(w)_X| = \ell(w) - \ell(u) = \ell(v)$. 

    This leaves the possibility that $w$ is one of the elements listed in
    Theorem \ref{T:chainbpexist}. If $w$ is the maximal element of $D_n$,
    $E_n$, or $F_4$, then as mentioned in the introduction the proposition has
    been proved by Barakat and Cuntz \cite{BC12}. As in the previous paragraph,
    the proposition holds for $w_{kl}$ if and only if it holds for
    $w_{kl}^{-1}$. The proof is finished by checking the elements $w_{kl}$ on
    a computer for $5 \leq l < k \leq 8$ (see the remarks in Subsection
    \ref{SS:computer}).
\end{proof}
\begin{example}
    To illustrate the proof, consider the element $w = w_1 s_5 s_4 s_2$ in $D_5$
    from Example \ref{Ex:typeD}, where $w_1$ is the maximal element in $W_J$ and
    $J = \{s_1,\ldots,s_4\}$. This element has a chain BP decomposition $w =
    w_1 v$ with respect to $J$, where $v = s_5 s_4 s_2$. So $X = \bigcap \mcI(w_1)$
    is a modular flat, and $\mcI(w)$ is free if and only if $\mcI(w)_X =
    \mcI(w_1)$ is free. But $\mcI(w_1)$ is the Coxeter arrangement of type $D_4$,
    albeit embedded in $\C^5$ rather than in $\C^4$. Thus $\mcI(w_1)$ is free
    with coexponents $0, 1, 3, 3, 5$. Since $\ell(v) = 3$, Lemma
    \ref{L:modularfree} states that $\mcI(w)$ is free with coexponents $1, 3,
    3, 3, 5$.  
\end{example}
Although it's not necessary to our argument, the result of applying Theorem
\ref{T:bpmodular} and Lemma \ref{L:modularfree} in the proof of Theorem
\ref{T:main1} can be explained in terms of the group $W$. Suppose $w$ has a
chain BP decomposition $w = uv$ with respect to $J$, where $u$ satisfies the
HLSS condition. Since $[e,v] \cap {}^J W$ is a chain, $v$ has a unique right
descent $s$, and $[e,vs] \cap {}^J W$ is still a chain. Let $\alpha = - w
\alpha_{s} \in I(w)$ be the inversion corresponding to the appearance of $s$ in
the last position of a reduced expression for $w$. If $\beta \in u I(v)$, then
as in the proof of Theorem \ref{T:bpmodular} there is $\gamma \in I(u)$ such
that $\alpha, \beta, \gamma$ are linearly dependent.  Hence if $H = \ker
\alpha$, then $\mcI(w)^H = \mcI(u)$, while $\mcI(w) \setminus H = \mcI(ws)$.
The element $ws$ has length one less than $w$, and chain BP decomposition $w =
u (vs)$, so applying the addition theorem repeatedly we find that if $\mcI(u)$
is free, then $\mcI(w)$ is free, and the coexponents of $\mcI(w)$ are the
coexponents of $\mcI(u)$ along with $\ell(v)$.  

\subsection{Inductive freeness of the elements $w_{kl}$}\label{SS:computer}
To check if an arrangement $\mcA$ is inductively free, we go through the
hyperplanes $H \in \mcA$ and see if $\mcA \setminus H$ and $\mcA^H$ are
inductively free. We refer to this procedure as the naive algorithm. We refer
to the hyperplane $H$ used at each step of this procedure as the \emph{pivot}.
We say that $H \in \mcA$ is a \emph{good pivot} if $\mcA^H$ is inductively free
with exponents $m_1,\ldots,m_{l-1}$, and $\mcA \setminus H$ is inductively free
with exponents $m_1,\ldots,m_l-1$. Using this terminology, the naive algorithm
can be described as follows:
\begin{enumerate}[(1)]
    \item If $\mcA$ is empty, then $\mcA$ is inductively free with all
        exponents zero.
    \item Otherwise, for each $H \in \mcA$:
        \begin{enumerate}[(i)]
            \item Determine if $\mcA^H$ and $\mcA\setminus H$ are inductively
                free. 
            \item If $H$ is a good pivot, stop: $\mcA$ is inductively free, and
                the exponents can be determined from the exponents of $\mcA^H$
                and $\mcA \setminus H$.
        \end{enumerate}
    \item If we do not find a good pivot in step 2, then $\mcA$ is not
        inductively free.  
\end{enumerate}
Unfortunately, the number of hyperplanes in the arrangements $\mcI(w_{kl})$,
shown in Table \ref{TBL:exceptlen}, make it impossible to apply the above naive
algorithm to the arrangements $\mcI(w_{kl})$.
\begin{table}[h]
    \begin{tabular}{lc}
        $w\quad \quad$ & $\ell(w) = |\mcI(w)|$ \\
        $w_{65}$ & 28 \\
        $w_{75}$ & 38 \\
        $w_{85}$ & 50 \\
        $w_{76}$ & 46 \\
        $w_{86}$ & 58 \\
        $w_{87}$ & 75 \\
    \end{tabular}
    \caption{Length of the elements $w_{kl}$.}
    \label{TBL:exceptlen}
\end{table}
To show that the arrangements $\mcI(w_{kl})$ are inductively free, we make
three improvements on the naive algorithm:
\begin{itemize}
    \item Memoization: Whenever we determine the inductive
        freeness of an arrangement, we record the result so we do not have to
        traverse through the induction tree a second time.
    \item Heuristic for pivots: If $\mcA$ is inductively free, then the Poincare
        polynomial $Q_{\mcA}(t)$ splits, and the exponents can be recovered from
        the roots. The Poincare polynomial of an arrangement can be computed
        relatively quickly, so at each stage of the recursion we keep track of
        the roots of $Q_{\mcA}(t)$. Before checking the inductive freeness
        of $\mcA^H$ and $\mcA \setminus H$, we first make sure that their
        Poincare polynomials split, and that the roots have the correct form. 
    \item Terminate early: Since all arrangements of rank $\leq 2$ are
        inductively free, we stop the recursion at rank $2$.
\end{itemize}
Using a C++ program incorporating these features, we can show that
$\mcI(w_{87})$ is inductively free in under 30 minutes. However, memoization
requires around $26$ gigabytes of random access memory (RAM). It is likely that
this computation time can be improved. Barakat and Cuntz report times of under
5 minutes for showing that $E_8$ is inductively free, using a good heuristic
initial ordering of the roots \cite{BC12}.

The complexity of the above approach makes it difficult to guarantee that our
program does not have a bug. As in \cite{BC12}, we address this problem by
using the complicated program outlined above to generate a certificate of inductive
freeness for each arrangement. Specifically, we say that a string $C$ is a
certificate of inductive freeness for $\mcA$ if $C$ is empty when the 
effective rank of $\mcA$ is $\leq 2$, and otherwise is of the form
$[H,C_1,C_2]$, where $H$ is a good pivot for $\mcA$, $C_1$ is a certificate for
$\mcA \setminus H$, and $C_2$ is a certificate for $\mcA^H$. The validity of a
certificate can be checked with a much simpler program (although the
computation time can be longer, since no memoization is used). Certificates
for the elements $w_{kl}$ are available on the author's webpage. The largest
certificate file is 50 megabytes (MB) ($<$ 5 MB when compressed).

\section{The flattening map}\label{S:flattening}

In this section, we recall the definition of the flattening map, and show that
the HLSS condition and freeness are preserved by flattening.  Let $U$ be a
subspace of $V$. The intersection $R_U = R \cap U$ is also a root
system, with positive and negative roots $R^+_U = R^+ \cap U$ and $R_U^- = R^-
\cap U$ respectively. Let $W_U$ be the Weyl group of $R_U$.  The Weyl group
$W_U$ is isomorphic to the parabolic subgroup $\langle t_{\beta} : \beta \in
R_U \rangle$ of $W$, or equivalently can be regarded as the subgroup of $W$
which acts identically on the orthogonal complement of $U$ inside of $V$
\cite{Stein64}. A subset $I$ of $R^+$ is convex if $\alpha,\beta \in I$,
$\alpha + \beta \in R^+$ implies that $\alpha + \beta \in I$. A subset $I$ of
$R^+$ is coconvex if the complement $R^+ \setminus I$ is convex.  Finally, a
subset $I$ is biconvex if it is both convex and coconvex. A subset of $R^+$ is
biconvex if and only if it is an inversion set $I(w)$ for some $w \in W$. Since
biconvexity is a linear condition, the intersection $I(w) \cap U$ is biconvex
in $R_U$, and thus there is an element $w' \in W_U$ such that $I(w') = I(w)
\cap U$. The element $w'$ is called the flattening of $w$, and will be denoted
by $\fl_U(w)$ \cite{BP05}. If $U = V_J = \vspan \{ \alpha_s : s \in J\}$, then
$\fl_U(w) = u$, where $w = uv$ is the left parabolic decomposition of $w$ with
respect to $J$. We use the following key property of flattening:
\begin{lemma}[\cite{BB03}]\label{L:flatequiv}
    If $u \in W_U$, $w \in W$, then $\fl_U(uw) = u \fl_U(w)$. 
\end{lemma}
Recall from the beginning of Section \ref{S:hultman} that a convex order on an
inversion set $\mcI(w)$ is an order coming from a reduced expression. An
arbitrary total order $\prec$ on $I(w)$ is convex if and only if it satisfies
two conditions \cite{Pa94}:
\begin{itemize}
    \item if $\alpha \prec \beta \in I(w)$, and $\alpha + \beta \in I(w)$,
        then $\alpha \prec \alpha+\beta \prec \beta$.
    \item if $\alpha \in I(w)$, $\beta \not\in I(w)$, and $\alpha - \beta
        \in I(w)$, then $\alpha-\beta \prec \alpha$.
\end{itemize}
Because these conditions are linear, we immediately get the following lemma.
\begin{lemma}
    If $\prec$ is a convex order on $I(w)$, then the induced order on
    $I(\fl_U(w)) = I(w) \cap U$ is also convex.
\end{lemma}

This leads to the main result of this section:
\begin{prop}\label{P:hultmanflatten}
    Let $U \subset V$ be any subspace.  If $w$ satisfies the HLSS condition,
    then so does $\fl_U(w)$ 
\end{prop}
\begin{proof}
    Choose a convex order $\prec$ on $I(w)$, and take the induced convex order
    on $I(\fl_U(w))$. If $x \in \left[e,\fl_U(w)\right]$, we can always find $u
    = t_{\beta_1} \cdots t_{\beta_m}$, where $\beta_1 \prec \cdots \prec
    \beta_m$ in $I(\fl_U(w))$, such that $x = u\fl_U(w)$. We want to show that
    we can take $\{\beta_1,\ldots,\beta_m\}$ to be an $\nbc$-set with respect to
    the given convex order. Now $x' = u w$ is less than $w$ in the Bruhat order,
    and since the HLSS condition is satisfied for $w$ we can find an $\nbc$-set
    $\{\gamma_1,\ldots,\gamma_k\}$ such that $u = t_{\gamma_1} \cdots t_{\gamma_k}$,
    where $\gamma_1 \prec \cdots \prec \gamma_k$. Let $X$ denote the fixed point
    subspace of $u$. By Theorem \ref{T:carter}, $X$ is equal to 
    $\vspan\{\gamma_1,\ldots,\gamma_k\}^\perp$ and contains
    $\vspan\{\beta_1,\ldots,\beta_m\}^\perp$. It follows that
    $\vspan\{\gamma_1,\ldots,\gamma_k\} \subset \vspan\{\beta_1,\ldots,\beta_m\} \subset U$, and 
    hence $\{\gamma_1,\ldots,\gamma_k\}$ is an $\nbc$-set in $I(\fl_U(w))$ as
    desired. We conclude that the restriction of the map $\phi : 2^{\mcI(\fl_U(w))} \arr
    [e,\fl_U(w)]$ to $\nbc(\mcI(\fl_U(w)))$ is surjective. 
\end{proof}

Now we turn to freeness.  For the inversion hyperplane arrangement, flattening
is the same as localization up to a quotient.
\begin{lemma}\label{L:local}
    Suppose $U \subset V$, and let $U^\perp$ be the orthogonal complement to $U$. 
    Given $w \in W$, let
    \begin{equation*}
        X = \bigcap_{\alpha \in I(w) \cap R_U} \ker\alpha,
    \end{equation*}
    so that $U^\perp \subset X$.  Then the inversion hyperplane arrangement
    $\mcI(\fl_U(w))$ is equal to the localization $\mcI(w)_X / U^\perp$.
\end{lemma}
If $U$ is spanned by $I(w) \cap U$, then $U^\perp$ is equal to $X$,
and $\mcI(\fl_U(w))$ is simply the quotient of $\mcI(w)_X$ by its center.
Conversely, given $X \in L(\mcI(w))$, if we let $U = \vspan\{\alpha \in I(w) :
X \subset \ker \alpha\}$ then $\mcI(\fl_U(w)) \iso \mcI(w)_X / X$. 

It is well-known that localization preserves freeness:
\begin{thm}[\cite{OT92}, Theorem 4.37]\label{T:localfree}
    If $\mcA$ is a free arrangement, and $X \in L(\mcA)$, then
    $\mcA_X$ is also free.
\end{thm}

As an immediate consequence of Theorem \ref{T:localfree} and Lemma
\ref{L:local} we have:
\begin{cor}\label{C:flatfree}
    If $\mcI(w)$ is free, then so is $\mcI(\fl_U(w))$.
\end{cor}

Localization also preserves supersolvability:
\begin{lemma}\label{L:supersolvflat}
    If $\mcI(w)$ is supersolvable and $U \subset V$ then $\mcI(\fl_U(w))$
    is supersolvable.
\end{lemma}
\begin{proof}
    By a theorem of Bj\"orner and Ziegler \cite{BZ91}, an arrangement $\mcA$
    is supersolvable if and only if there is an order on the set $\{\alpha_H :
    H \in \mcA\}$ of defining forms such that every minimal broken circuit with
    respect to inclusion has size $2$. 

    Given an order with this property on $I(w)$, take the induced order on
    $I(w) \cap U$. Since $\mcI(w) \cap U$ is linearly closed, a subset of $I(w)
    \cap U$ is a broken circuit if and only if it is a broken circuit in
    $I(w)$. Hence all minimal broken circuits in $I(w) \cap U$ have size $2$,
    and $\mcI(\fl_U(w))$ is supersolvable.
\end{proof}

\subsection{Proof of main results}\label{SS:proof}

Let $R$ and $R'$ be two root systems with Weyl groups $W(R)$ and $W(R')$
respectively, and let $w' \in W(R')$. An element $w \in W(R')$ is said to
contain the pattern $(w',R')$ if there is a subspace $U$ of the ambient space of
$R$ such that $R_{U}$ is isomorphic to $R'$, and $\fl_U(w) = w'$ when $R_{U}$
is identified with $R'$. If this does not happen for any subspace $U$, then $w$
is said to avoid $(w',R')$. This notion of root system pattern avoidance due to
Billey and Postnikov \cite{BP05} generalizes the usual notion of pattern
avoidance for permutations. 

The main result of \cite{BP05} is that an element $w \in W$ is rationally
smooth if and only if it avoids a finite list of patterns in the classical
types $A_3$, $B_3$, $C_3$, and $D_4$. In Table \ref{TBL:badpat} below, we list
each bad pattern $w'$, along with the Poincare polynomial $Q_{\mcI(w')}(t)$
(shown in factored form) of the inversion arrangement, the number
$Q_{\mcI(w')}(1)$ of $\nbc$-sets of the inversion arrangement, and the size of
the Bruhat interval $|[e,w']|$. 
\begin{table}[h]
    \begin{tabular}{lllll}
        $R'$ & $w'$ & $Q_{\mcI(w')}(t)\quad\quad$ & $|\nbc(\mcI(w'))|\quad$ & $|[e,w']|$ \\
        $A_3$ & $s_2 s_1 s_3 s_2$ & $(1+t)(1+ 3t + 3t^2)$ & 14 & 14 \\
        $A_3$ & $s_1 s_2 s_3 s_2 s_1$ & $(1+t)(1+2t)^2$ & 18 & 20 \\
        $D_4$ & $s_2 s_1 s_3 s_4 s_2$ & $(1+t)(1+2t)(1+2t+2t^2)$ & 30 & 30 \\
        $B_3/C_3$ & $s_2 s_1 s_3 s_2$ & $(1+t)(1+3t+3t^2)$ & 14 & 14 \\ 
        $B_3/C_3$ & $s_3 s_2 s_1 s_3 s_2$ & $(1+t)(1+4t+5t^2)$ & 20 & 20 \\ 
        $B_3/C_3$ & $s_2 s_1 s_3 s_2 s_3$ & $(1+t)(1+4t+5t^2)$ & 20 & 20 \\
        $B_3/C_3$ & $s_3 s_2 s_1 s_3 s_2 s_3$ & $(1+t)(1+5t+7t^2)$ & 26 & 26 \\
        $B_3/C_3$ & $s_3 s_2 s_1 s_2 s_3$ & $(1+t)(1+2t)^2$ & 18 & 20 \\
        $B_3/C_3$ & $s_2 s_3 s_2 s_1 s_2 s_3$ & $(1+t)(1+5t+7t^2)$ & 26 & 28 \\
        $B_3/C_3$ & $s_3 s_2 s_1 s_2 s_3 s_2$ & $(1+t)(1+5t+7t^2)$ & 26 & 28 \\
        $B_3/C_3$ & $s_2 s_3 s_2 s_1 s_2 s_3 s_2$ & $(1+t)(1+6t+10t^2)$ & 34 & 36 \\
        $B_3/C_3$ & $s_1 s_2 s_3 s_2 s_1$ & $(1+t)(1+2t)^2$ & 18 & 20 \\
        $B_3/C_3$ & $s_1 s_2 s_3 s_2 s_1 s_3$ & $(1+t)(1+2t)(1+3t)$ & 24 & 28 \\
        $B_3/C_3$ & $s_1 s_2 s_3 s_2 s_1 s_2 s_3$ & $(1+t)(1+3t)^2$ & 32 & 36 \\ 
        $B_3/C_3$ & $s_1 s_2 s_3 s_2 s_1 s_3 s_2$ & $(1+t)(1+3t)^2$ & 32 & 36 \\
        $B_3/C_3$ & $s_1 s_2 s_3 s_2 s_1 s_2 s_3 s_2$ & $(1+t)(1+3t)(1+4t)$ & 40 & 42 \\
        $B_3/C_3$ & $s_1 s_2 s_3 s_2 s_1 s_3 s_2 s_3$ & $(1+t)(1+3t)(1+4t)$ & 40 & 44 \\ 
    \end{tabular}
    \caption{Patterns characterizing rational smoothness.}
    \label{TBL:badpat}
\end{table}
Since rational smoothness depends only on the Coxeter system, not on the choice
of root system, the patterns for types $B_3$ and $C_3$ come in pairs $(w',
B_3)$ and $(w',C_3)$. If we take $R = B_n$, then $C_n$ is the dual root system
$R^{\vee} = \{\alpha^{\vee}: \alpha \in R\}$, where $\alpha^{\vee} = 2 \alpha /
(\alpha,\alpha)$. From this fact, it is not hard to see that the inversion
arrangement of an element $w \in W(B_n)$ is isomorphic to the inversion
arrangement of the same element in $W(C_n)$. Thus we include only one line for
each pair of patterns in $B_3$ and $C_3$.  The elements $w'$ are written using
the following Dynkin diagram labelling, taken from \cite{BP05}:
\begin{align*}
    A_3:&\quad\xymatrix{ 1 \ar@{-}[r] & 2 \ar@{-}[r] & 3} \\
    B_3/C_3:&\quad\xymatrix{ 1 \ar@{-}[r] & 2 \ar@2{-}[r] & 3}  \\
    D_4:&\quad \vcenter{\xymatrix{              & 3 \ar@{-}[d] &          \\
                                1 \ar@{-}[r] & 2 \ar@{-}[r] & 4  }}
\end{align*}

\begin{proof}[Proof of Theorem \ref{T:main}]
    Suppose $w$ is rationally smooth. Then by Theorem \ref{T:main1}, 
    $\mcI(w)$ is free with coexponents equal to the exponents of $\mcI(w)$.
    It follows that $Q_{\mcI(w)}(t) = \prod_i (1+m_i t)$, so the number of
    $\nbc$-sets is $Q_{\mcI(w)}(1) = \prod_i (1+m_i) = P_w(1) = |[e,w]|$.

    Conversely, suppose that $\mcI(w)$ is free with coexponents $d_1,\ldots,
    d_l$, and $\prod_i (1+d_i) = |[e,w]|$. The product $\prod_i (1+d_i)$ is the
    number of $\nbc$-sets, so $w$ satisfies the HLSS condition. By Proposition
    \ref{P:hultmanflatten} and Corollary \ref{C:flatfree}, if $w$ contains
    $(w', R')$ then $\mcI(w')$ is free and $w'$ satisfies the HLSS condition.
    For every element $w'$ listed in Table \ref{TBL:badpat}, either the
    Poincare polynomial $Q_{\mcI(w')}(t)$ does not factor into linear factors
    over $\Z$, implying that $\mcI(w')$ is not free, or the number of
    $\nbc$-sets is less than $|[e,w']|$, implying that $w'$ does not satisfy
    the HLSS condition. Hence if $\mcI(w)$ is free and $w$ satisfies the HLSS
    condition, then $w$ cannot contain any of the patterns listed in Table
    \ref{TBL:badpat}, so $w$ is rationally smooth.
\end{proof}

\begin{proof}[Proof of Theorem \ref{T:supersolv}]
    As mentioned in Remark \ref{R:supersolv}, if $w$ has a complete chain
    decomposition, then $\mcI(w)$ is supersolvable. Conversely, suppose that
    $w$ is rationally smooth and $\mcI(w)$ is supersolvable. The inversion
    arrangements of the maximal elements of type $D_4$ and $F_4$ are not
    supersolvable \cite{HR13}. If $w$ is the maximal elements of type $D_n$ or
    $E_n$, or one of the elements $w_{kl}^{-1}$, $5 \leq l < k \leq 8$, then
    there is a subset $J \subset S$ such that $\fl_{V_J}(w)$ is the maximal element
    in $D_4$. By Lemma \ref{L:supersolvflat}, this cannot happen if $\mcI(w)$
    is supersolvable. Since $\mcI(w_{kl})$ is linearly equivalent to
    $\mcI(w^{-1}_{kl})$, $w$ cannot be one of the elements $w_{kl}$. By Theorem
    \ref{T:chainbpexist}, we conclude that $w$ has a right or left chain BP
    decomposition $w = vu$ or $w = uv$ with respect to some $J$.  Now $u$ is
    rationally smooth and $\mcI(u)$ is supersolvable by Lemma
    \ref{L:supersolvflat}, so we can repeat this process with $u$ to show that
    $w$ has a complete chain BP decomposition.
\end{proof}

\section{Further directions}\label{S:further}
The positive roots $R^+$ of $R = A_n$ can be realized as $\{e_i - e_j : 1 \leq
i < j \leq n+1\}$, where $e_1,\ldots,e_{n+1}$ is a standard basis of $\R^n$.
Subsets of $R^+$ can be identified with graphs on vertex set $1,\ldots,n+1$ by
identifying $e_i - e_j$ with the edge $ij$. For this reason, a union of
hyperplanes $H = \ker \alpha$, $\alpha$ a positive root of $A_n$, is called a
graphic arrangement. A theorem of Stanley states that a graphic arrangement is
free if and only if the corresponding graph is chordal, meaning that every
cycle of length $\geq 4$ has a chord \cite{ER94}.
\begin{lemma}\label{L:chordless}
    If $I$ is a biconvex set in $A_n$, then every chordless cycle in the
    corresponding graph has length $\leq 4$. 
\end{lemma}
\begin{proof}
    Suppose $C$ is a cycle of length $\geq 5$ in the graph $G$ corresponding to
    $I$. Let $i$ be the smallest vertex in $C$, and let $j < k$ be the two
    vertices adjacent to $i$ in $C$. Finally, let $m \neq k$ be the other
    vertex adjacent to $j$ in $C$. Since $C$ has length $\geq 5$, $m,j,i,k$ are
    all distinct, and any edge between these vertices which is not already in
    $C$ will be a chord. 

    If $m > j$, then the root corresponding to edge $jm$ is $e_j - e_m$, so
    $e_j - e_m \in I$. Similarly $e_i - e_j \in I$. Since $i<m$, $e_i - e_m =
    (e_i - e_j) + (e_j - e_m)$ is a positive root, and since $I$ is biconvex,
    $e_i - e_m \in I$. Hence $C$ has a chord in $G$. 
    
    If $m < j$, then $e_m - e_k$ is a positive root. Since $I$ is coconvex and 
    $e_i - e_k = (e_m - e_k) + (e_i - e_m)$ belongs to $I$, one of $e_m - e_k$
    or $e_i - e_m$ must belong to $I$, and again $C$ will have a chord in $G$.
\end{proof}
If $R$ is of type $A$, then any subsystem $R_U$ also has type $A$ (or is a product
of subsystems of type $A$). Hence pattern avoidance conditions in type $A$ are
usually written in terms of permutations.  For example, Billey and Postnikov's
root system pattern avoidance criterion simplifies in type $A$ to the earlier
pattern avoidance criterion of Lakshmibai and Sandhya:
\begin{thm}[\cite{LS90}]\label{T:typeArs}
    If $w$ is a permutation in $A_n$, then $w$ is rationally smooth if and only
    if $w$ avoids the permutations $3412$ and $4231$.
\end{thm}

Using Lemma \ref{L:chordless}, it is easy to give a pattern characterization of
freeness in type $A$:
\begin{prop}\label{P:typeAfree}
    If $w$ is a permutation in $A_n$, then $\mcI(w)$ is free if
    and only if $w$ avoids the permutation $3412$.
\end{prop}
\begin{proof}
    Let $G$ be the graph corresponding to $I(w)$.  By Stanley's theorem and
    Lemma \ref{L:chordless}, $\mcI(w)$ is free if and only if every
    $4$-cycle in $G$ has a chord. If $U$ is a subspace of the ambient space,
    let $G_U$ denote the subgraph of $G$ with edge set $\{ij : e_i - e_j \in
    I(w) \cap U\}$. Up to some disconnected vertices, $G_U$ is the graph
    corresponding to $I(\fl_U(w))$. If $C$ is a $4$-cycle, then the space $U =
    \vspan\{e_i - e_j : ij \in C\}$ has dimension $3$, so $\mcI(w)$ is free if
    and only if $\mcI(\fl_U(w))$ is free for all subspaces $U$ of rank $3$. 

    The inversion set $I(3412) = \{e_2-e_3, e_1-e_3, e_2-e_4, e_1-e_4\}$, and
    the corresponding graph is the $4$-cycle $2-3-1-4$. It is not hard to check
    that this is the only inversion set in $A_3$ with a chordless $4$-cycle.
\end{proof}
Note that the reduced word for $3412$ is $s_2 s_1 s_3 s_2$. Proposition
\ref{P:typeAfree} makes it clear that the freeness of $\mcI(w)$ is not
equivalent to rational smoothness. For example, in $A_3$ the permutation $4231$
has reduced word $s_1 s_2 s_3 s_2 s_1$ and inversion set $\{e_1 - e_2, e_3 -
e_4, e_1 - e_3, e_2 - e_4, e_1-e_4\}$, which corresponds to the complete graph
minus an edge.  Hence $\mcI(4231)$ is free with exponents $1, 2, 2$ (see Table
\ref{TBL:badpat}).

In type $A$, the HLSS condition can also be can also be characterized by
pattern avoidance:
\begin{thm}[\cite{HLSS09}]\label{T:typeAHLSS}
    If $w$ is a permutation in type $A_n$, then $w$ satisfies the HLSS
    condition if and only if $w$ avoids $4231$, $35142$, $42513$, and $351624$.
\end{thm}
The patterns $35142$, $42513$, and $351624$ all contain the pattern $3412$. 
By combining Theorem \ref{T:typeArs}, Proposition \ref{P:typeAfree}, and
Theorem \ref{T:typeAHLSS}, we get another proof that $w$ is rationally smooth
if and only if $\mcI(w)$ is free and $w$ satisfies the HLSS condition.
However, this approach does not show that the coexponents and exponents agree.

Proposition \ref{P:hultmanflatten} and Corollary \ref{C:flatfree} suggest that
it should be possible to extend Proposition \ref{P:typeAfree} and Theorem
\ref{T:typeAHLSS} to arbitrary finite type. The problem of finding a pattern
avoidance criterion for freeness is addressed in \cite{Sl13b}. This
leaves:
\begin{problem}
    Find a root system pattern avoidance criterion for the HLSS condition
    in arbitrary finite type.
\end{problem}

Finally, we have the following question:
\begin{question}
    Do Theorems \ref{T:exponents}, \ref{T:main} and \ref{T:supersolv} hold in
    arbitrary (finite) Coxeter groups? In particular, does the Poincare
    polynomial $Q_{\mcI(w)}(t)$ have a linear factorization when $P_w(q)$ is
    palindromic?
\end{question}

\bibliographystyle{amsalpha}
\bibliography{palindromic}

\providecommand{\bysame}{\leavevmode\hbox to3em{\hrulefill}\thinspace}
\providecommand{\MR}{\relax\ifhmode\unskip\space\fi MR }
\providecommand{\MRhref}[2]{%
  \href{http://www.ams.org/mathscinet-getitem?mr=#1}{#2}
}
\providecommand{\href}[2]{#2}
\begin{thebibliography}{HLSS09}

\bibitem[AC12]{AC12}
Ersan Akyildiz and James~B. Carrell, \emph{Betti numbers of smooth {S}chubert
  varieties and the remarkable formula of {K}ostant, {M}acdonald, {S}hapiro,
  and {S}teinberg}, Michigan Math. J. \textbf{61} (2012), no.~3, 543--553.

\bibitem[Arn79]{Ar79}
V.~I. Arnold, \emph{Indices of singular points of 1-forms on a manifold with
  boundary, convolution of invariants of reflection groups, and singular
  projections of smooth surfaces}, {R}ussian {M}ath. {S}urveys \textbf{34}
  (1979), 1--42.

\bibitem[BB03]{BB03}
Sara~C. Billey and Tom Braden, \emph{Lower bounds for {Kazhdan}-{Lusztig}
  polynomials from patterns}, Transformation Groups \textbf{8} (2003), no.~4,
  321--332.

\bibitem[BB05]{BB05}
Anders Bj{\"o}rner and Francesco Brenti, \emph{Combinatorics of {C}oxeter
  groups}, Graduate Texts in Mathematics, vol. 231, Springer, New York, 2005.

\bibitem[BC12]{BC12}
Mohamed Barakat and Michael Cuntz, \emph{Coxeter and crystallographic
  arrangements are inductively free}, Advances in Mathematics \textbf{229}
  (2012), no.~1, 691--709.

\bibitem[Bil98]{Bi98}
Sara~C. Billey, \emph{Pattern avoidance and rational smoothness of {S}chubert
  varieties}, Adv. Math. \textbf{139} (1998), no.~1, 141--156.

\bibitem[BL00]{BL00}
Sara Billey and V.~Lakshmibai, \emph{Singular loci of {S}chubert varieties},
  Progress in Mathematics, vol. 182, Birkh\"auser Boston Inc., Boston, MA,
  2000.

\bibitem[BP99]{BP99}
Michel Brion and Patrick Polo, \emph{Generic singularities of certain
  {S}chubert varieties}, Math. Z. \textbf{231} (1999), no.~2, 301--324.

\bibitem[BP05]{BP05}
Sara Billey and Alexander Postnikov, \emph{Smoothness of {S}chubert varieties
  via patterns in root subsystems}, Adv. in Appl. Math. \textbf{34} (2005),
  no.~3, 447--466.

\bibitem[BZ91]{BZ91}
Anders Bj{\"o}rner and G{\"u}nter~M Ziegler, \emph{Broken circuit complexes:
  Factorizations and generalizations}, Journal of Combinatorial Theory, Series
  B \textbf{51} (1991), no.~1, 96--126.

\bibitem[Car72]{Ca72}
R.~W. Carter, \emph{Conjugacy classes in the {W}eyl group}, Compositio Math.
  \textbf{25} (1972), 1--59.

\bibitem[CK03]{CK03}
James~B. Carrell and Jochen Kuttler, \emph{Smooth points of {$T$}-stable
  varieties in {$G / B$} and the {P}eterson map}, Inventiones Mathematicae
  \textbf{151} (2003), no.~2, 353--379.

\bibitem[Deo77]{De77}
Vinay~V. Deodhar, \emph{Some characterizations of {B}ruhat ordering on a
  {C}oxeter group and determination of the relative {M}{\"o}bius function},
  Inventiones mathematicae \textbf{39} (1977), no.~2, 187--198.

\bibitem[ER94]{ER94}
Paul~H. Edelman and Victor Reiner, \emph{Free hyperplane arrangements between
  {$A_{n−-1}$} and {$B_n$}}, Mathematische Zeitschrift \textbf{215} (1994),
  no.~1, 347--365.

\bibitem[Gas98]{Ga98}
Vesselin Gasharov, \emph{Factoring the {P}oincar\'e polynomials for the
  {B}ruhat order on {$S_n$}}, J. Combin. Theory Ser. A \textbf{83} (1998),
  no.~1, 159--164.

\bibitem[GR02]{GR02}
V.~Gasharov and V.~Reiner, \emph{Cohomology of smooth {S}chubert varieties in
  partial flag manifolds}, J. London Math. Soc. (2) \textbf{66} (2002), no.~3,
  550--562.

\bibitem[HLSS09]{HLSS09}
Axel Hultman, Svante Linusson, John Shareshian, and Jonas Sj{\"o}strand,
  \emph{From {B}ruhat intervals to intersection lattices and a conjecture of
  {P}ostnikov}, J. Combin. Theory Ser. A \textbf{116} (2009), no.~3, 564--580.

\bibitem[HR]{HR13}
Torsten Hoge and Gerhard R{\"o}hrle, \emph{Supersolvable reflection
  arrangements}, Proceedings of the American Mathematical Society, to appear in
  print.

\bibitem[Hul11]{Hu11}
Axel Hultman, \emph{Inversion arrangements and {B}ruhat intervals}, J. Combin.
  Theory Ser. A \textbf{118} (2011), no.~7, 1897--1906.

\bibitem[LM14]{LM14}
Joel~B. Lewis and Alejandro~H. Morales, \emph{Combinatorics of diagrams of
  permutations}, preprint (2014), arXiv:1405.1608.

\bibitem[LS90]{LS90}
V.~Lakshmibai and B.~Sandhya, \emph{Criterion for smoothness of {S}chubert
  varieties in {${\rm Sl}(n)/B$}}, Proc. Indian Acad. Sci. Math. Sci.
  \textbf{100} (1990), no.~1, 45--52.

\bibitem[OPY08]{OPY08}
Suho Oh, Alexander Postnikov, and Hwanchul Yoo, \emph{Bruhat order, smooth
  {S}chubert varieties, and hyperplane arrangements}, J. Combin. Theory Ser. A
  \textbf{115} (2008), no.~7, 1156--1166.

\bibitem[OT92]{OT92}
Peter Orlik and Hiroaki Terao, \emph{Arrangements of hyperplanes}, Springer,
  July 1992.

\bibitem[OY10]{OY10}
Suho Oh and Hwanchul Yoo, \emph{Bruhat order, rationally smooth {S}chubert
  varieties, and hyperplane arrangements}, DMTCS Proceedings \textbf{FPSAC}
  (2010).

\bibitem[Pap94]{Pa94}
Paolo Papi, \emph{A characterization of a special ordering in a root system},
  Proceedings of the American Mathematical Society \textbf{120} (1994), no.~3,
  661--665.

\bibitem[RR13]{RR13}
Victor Reiner and Yuval Roichman, \emph{Diameter of graphs of reduced words and
  galleries}, Transactions of the American Mathematical Society \textbf{365}
  (2013), no.~5, 2779--2802.

\bibitem[RS14a]{RS13b}
Edward Richmond and William Slofstra, \emph{{B}illey-{P}ostnikov decompositions
  and the fibre bundle structure of {S}chubert varieties}, preprint (2014),
  arXiv:1408.0084.

\bibitem[RS14b]{RS13a}
\bysame, \emph{Rationally smooth elements of {C}oxeter groups and triangle
  group avoidance}, Journal of Algebraic Combinatorics \textbf{39} (2014),
  no.~3, 659--681.

\bibitem[Sai75]{Sa75}
K.~Saito, \emph{On the {U}niformization of {C}omplements of {D}iscriminant
  {L}oci}, Conference {N}otes ({A.M.S.} {S}ummer {I}nstitute), Amer. Math.
  Soc., Williamstown, 1975.

\bibitem[Sj{\"o}07]{Sj07}
Jonas Sj{\"o}strand, \emph{Bruhat intervals as rooks on skew {F}errers boards},
  Journal of Combinatorial Theory, Series A \textbf{114} (2007), no.~7,
  1182--1198.

\bibitem[Slo14]{Sl13b}
William Slofstra, \emph{Pattern avoidance criterion for free inversion
  arrangements}, preprint (2014), arXiv:1409.7299.

\bibitem[ST54]{ST54}
G.~C. Shephard and J.~A. Todd, \emph{Finite unitary reflection groups},
  Canadian Journal of Mathematics \textbf{6} (1954), no.~0, 274--304.

\bibitem[Ste64]{Stein64}
Robert Steinberg, \emph{Differential equations invariant under finite
  reflection groups}, Transactions of the American Mathematical Society
  \textbf{112} (1964), no.~3, 392--400.

\bibitem[Ter81]{Te81}
Hiroaki Terao, \emph{Generalized exponents of a free arrangement of hyperplanes
  and {S}hepherd-{T}odd-{B}rieskorn formula}, Inventiones mathematicae
  \textbf{63} (1981), no.~1, 159--179.

\bibitem[Yoo11]{Yoo11}
Hwanchul Yoo, \emph{Combinatorics in {S}chubert varieties and {S}pecht
  modules}, Thesis, Massachusetts Institute of Technology, 2011.

\end{thebibliography}

\end{document}